\documentclass{amsart}

\usepackage{amsmath} 
\usepackage{amssymb}
\usepackage{mathrsfs}

\newtheorem{theorem}{Theorem}[section] 
\newtheorem{claim}[theorem]{Claim}
\newtheorem{mc}[theorem]{Main Claim}

\theoremstyle{definition}
\newtheorem{definition}[theorem]{Definition}

\newtheorem{conjecture}[theorem]{Conjecture}
\newtheorem{discussion}[theorem]{Discussion}
\newtheorem{convention}[theorem]{Convention}

\newtheorem{observation}[theorem]{Observation} 

\newtheorem{hypothesis}[theorem]{Hypothesis}

\theoremstyle{remark}
\newtheorem{question}[theorem]{Question}
\newtheorem{remark}[theorem]{Remark}

\newcommand{\rest}{{\restriction}}
 
\newcommand{\tp}{{\rm tp}} 
\newcommand{\Th}{{\rm Th}} 
 
\newcommand{\pot}{{\rm pot}}
\newcommand{\com}{{\rm com}}

\newcommand{\sat}{{\rm sat}}
\newcommand{\AP}{{\rm AP}}
\newcommand{\PA}{{\rm PA}}
\newcommand{\wilog}{{\rm without loss of generality}}
\newcommand{\Wilog}{{\rm Without loss of generality}}

\newcommand{\then}{{\underline{then}}}
\newcommand{\when}{{\underline{when}}}
\newcommand{\Then}{{\underline{Then}}}

\newcommand{\Iff}{{\underline{iff}}}
\newcommand{\mn}{{\medskip\noindent}}
\newcommand{\sn}{{\smallskip\noindent}}

\newcommand{\bbZ}{{\mathbb Z}}

\newcommand{\bbL}{{\mathbb L}}
\newcommand{\bbN}{{\mathbb N}}

\newcommand{\bbP}{{\mathbb P}}

\newcommand{\bbQ}{{\mathbb Q}}
\newcommand{\bbR}{{\mathbb R}}

\newcommand{\cS}{{\mathscr S}}

\newcommand{\cf}{{\rm cf}}

\newcount\skewfactor
\def\mathunderaccent#1#2 {\let\theaccent#1\skewfactor#2
\mathpalette\putaccentunder}
\def\putaccentunder#1#2{\oalign{$#1#2$\crcr\hidewidth
\vbox to.2ex{\hbox{$#1\skew\skewfactor\theaccent{}$}\vss}\hidewidth}}

\newenvironment{PROOF}[2][\proofname.]
   {\begin{proof}[#1]\renewcommand{\qedsymbol}{\textsquare$_{\rm #2}$}}
   {\end{proof}}

\begin{document}

\title[Models of PA, etc.]
{Models of PA: when two elements are necessarily order automorphic}
\author {Saharon Shelah}
\address{Einstein Institute of Mathematics\\
Edmond J. Safra Campus, Givat Ram\\
The Hebrew University of Jerusalem\\
Jerusalem, 91904, Israel\\
 and \\
 Department of Mathematics\\
 Hill Center - Busch Campus \\ 
 Rutgers, The State University of New Jersey \\
 110 Frelinghuysen Road \\
 Piscataway, NJ 08854-8019 USA}
\email{shelah@math.huji.ac.il}
\urladdr{http://shelah.logic.at}
\thanks{This research was supported by the Israel Science 
Foundation. Publication 924.\\
I would like to thank Alice Leonhardt for the beautiful 
typing.}



\date{May 22, 2012}

\begin{abstract}
We are interested in the question of how much the order of a
non-standard model of PA can determine the model.  In particular, for a
model $M$, we want to characterize the complete types $p(x,y)$ of non-standard
elements $(a,b)$ such that the linear orders $\{x:x < a\}$ and $\{x:x
< b\}$ are necessarily isomorphic.  It is proved that this set includes the
complete types $p(x,y)$ such that if the pair $(a,b)$ realizes it (in
$M$) then there is an element $c$ such that for all
standard $n,c^n < a,c^n < b,a < bc$ and $b < ac$.  We prove that this
is optimal, because if $\diamondsuit_{\aleph_1}$ holds, then there is $M$
of cardinality $\aleph_1$ for which we get equality.  We also deal
with how much the order in a model of PA may determine the addition.
\end{abstract}

\maketitle
\numberwithin{equation}{section}
\setcounter{section}{-1}

\section {Introduction} 

Let $M$ be a model of Peano Arithmetic (PA).  For an $a \in M$, by
$M_{<a}$ we denote the set $\{c \in M:M \models c < a\}$ with the
inherited linear order.  For any
pair $(a,b)$ of non-standard elements of $M$, let $(*)_{M,a,b}$ be the
condition dfeined by
\mn
\begin{enumerate}
\item[$(*)_{M,a,b}$]  $(M_{<a},<_M) \cong (M_{<b},<_M)$.
\end{enumerate}
\mn
We will also consider
\mn
\begin{enumerate}
\item[$(*)^{\pot}_{M,a,b}$]  for every $N$, if $M \prec N$, then
  $(*)_{N,a,b}$
\sn
\item[$(*)^{\tp}_{M,a,b}$]  for every model $N$, if $M \equiv N$ and
for some $M_0 \prec N$ we have $\{a,b\} \subseteq M_0 \prec M$, then
$(*)_{N,a,b}$.
\end{enumerate}
\mn
Looking recently at models of PA, we wonder:
\begin{question}
\label{h.1}
What is the set of complete types $p(x,y)$ such that: if the pair
$(a,b)$ realizes the type $p(x,y)$ then $(*)_{M,a,b}$ holds?  
Another variant is, given a model $M$ and $a,b \in M$ when
do we have $(*)^{\pot}_{M,a,b}$ or just $(*)^{\tp}_{M,a,b}$?

Our main aim is sorting this out.  For the problem as stated, 
on the one hand we give a sufficient condition, and on the other hand, 
for $(*)^{\tp}_{M,a,b}$, we prove its necessity,
assuming $\diamondsuit_{\aleph_1}$.  However, we may consider a relative:
\end{question}

\begin{question}
\label{h.2}
Like question \ref{h.1}, but we restrict ourselves to $\aleph_1$-saturated
models.  
\end{question}

It seems natural to ask:
\begin{question}
\label{h.3}  
Generally, how much does the linear order of a non-standard
model $M$ of PA determine $M$?  Is there a non-standard model $M$ of
PA such that for every model $N$ of PA, if $(M \rest \{<\}) \cong (N,
\rest \{<\})$, \then \, are $M \cong N$? 
\end{question}

We discussed those problems with Gregory Cherlin and he asked:
\begin{question}
\label{h.5}
[Cherlin] Show that $\{M \rest \{<\}:M \models \PA\}$ is complicated.

This question is too vague for my taste.
\end{question}

\noindent
Recall the much earlier problem 14 from Kossak-Schmerl \cite{KoSc06}
asked by Friedman:
\begin{question}
\label{h.7}
Is there a model of $\PA$ such that for every model $N$ of $\PA$, if
$(M \rest \{<\}) \cong (N, \rest \{<\})$, then $M \equiv N$?

This seems of different character, speaking just of the theory of
another model, but of course, a positive answer to Question \ref{h.3}
would also give an answer to Question \ref{h.7}.
\end{question}

\noindent
We may go half way: maybe the linear order of $M$ does not determine
$M$, say up to isomorphism, but just the additive structure (from
which the order is definable).  This means
\begin{question}
\label{h.9}  
How much the order of $M$, a non-standard model of 
$T,T \in \text{ PA}^{\text{com}}$, determines the isomorphism 
type of $(|M|,<_M,+_M)$?
\end{question}

\noindent
A more general version of our question is
\begin{question}
\label{z10}
Can we construct a non-standard model $M$ of PA with few order
automorphisms in some sense?
\end{question}

Recall that for any countable non-standard model $M$ of $\PA$, it is
recursively saturated hence has ``lots" of order automorphisms (see
\cite{KoSc06}).  
Much has been done on other classes.  Concerning Abelian groups and
modules, see the book G\"obel-Trlifaj \cite{GbTl05}.  For general first
order see \cite{Sh:384} and history in both.  In particular, there are
non-standard models of PA with no automorphism, 
this motivating the ``order-automorphism"
in \ref{z10}.  Now answer to \ref{h.1} sheds some light on \ref{h.9}.

Let us review the present work.
First, in \S1 we introduce and deal with some relevant equivalence relations,
and in \ref{a13} it is proved that the so called $a E^2_M b$ implies
$(*)_{M,a,b}$ so $a E^2_M b$ is a sufficient condition for a
 positive answer to Question \ref{h.1}(1), while for the
so called 2-order rigid models $M$, we prove that the
isomorphism type of $M \rest \{<\}$ determines that of $M \rest \{<,+\}$
but only up to almost isomorphism, shedding some light on \ref{h.9}.

Second, in \S2 we get that even $a E^3_M b$ implies $(*)_{M,a,b}$.
This shows that \ref{a13} is not so interesting but its proof is a
warm-up for \S2.  Moreover this is only part of the picture, see \S4.  In \S2
we also show that if $M$ is 3-o.r. then $M \rest \{<,+\}$ is unique up
to almost isomorphism.

In \S3 we show that $E^3_M$ is the right notion as if
$\diamondsuit_{\aleph_1}$ holds then every countable model of $\PA$ 
has elementary extension $M$ of cardinality $\aleph_1$ such that for $a,b \in M
\backslash \bbN$ we have $a E^3_M b \Leftrightarrow M_{< a} \cong
M_{<b}$.  We comment there on the case $\neg a E^4_M b$.

Naturally, for most results some weaker version of PA suffices.  We
comment on this in \S4; so usually when a result supercedes an
earlier one, normally it has a harder proof and really use more axioms
of PA.  

We thank the referee for doing much to improve the presentation,
considerably more than the call of duty and Shimoni Garti for help in
proofreading. 

\begin{convention}
\label{z15}
Models are models of PA, even ordinary ones if not said otherwise
where $M$ is ordinary when $\bbN \subseteq M$.  It is used e.g. in
Definition \ref{a5}(b),(c).  We may circumvent it, defining $a +_M n$
by repeated addition of $1_M$ and $a \times_M n$ by repeated additions of
$a$, but it seems to me less convenient.
\end{convention}
\newpage

\section {Somewhat rigid order}

We define (in \ref{a5}) some equivalence relations $E^\ell_M$ for
models $M$ (of PA).  We shall deal with their basic properties in
\ref{a7}, \ref{a21}, the relations between them (in \ref{b3}),
cofinalities of equivalence classes (in \ref{a19}, \ref{b11}), on order 
isomorphism/almost $\{<,+\}$-isomorphism in \ref{a3},
\ref{a9}, \ref{b9}; including noting for $\ell$-o.r. models that
$E^\ell_M \subseteq E^5_M$.  

Lastly, we prove versions of ``if $M_1 \rest \{<\},M_2 \rest \{<\}$ are
 isomorphic then $M_1 \rest \{<,+\},M_2 \rest \{<,+\}$ are almost isomorphic",
see Theorems \ref{a13}, \ref{b13}.
\bigskip

\begin{definition}
\label{a3}  
We say $M,N$ are almost $\{<,+\}$-isomorphic \when \,: some $f$ is an
almost $\{<,+\}$-isomorphism from $M$ onto $N$ which means $f$ is an 
isomorphism  from $M \rest \{<\}$ onto $N \rest \{<\}$ such that
for all $a,b \in M$ there is an $n \in \bbN$ such that the distance
between $f(a +_M b)$ and $f(a) +_N f(b)$ is $n$.
\end{definition}

\begin{definition}
\label{a5}  
We define the following equivalence relations on $M \backslash \bbN$
\mn
\begin{enumerate}
\item[$(a)$]   $E^0_M:a E^0_M b$ \Iff \,  $\bigvee\limits_{n \in \bbN} 
(a <_M b + n \wedge b <_M a +n)$
\sn
\item[$(b)$]   $E^1_M:a E^1_M b$ \Iff \, for some $c  \in M$ satisfying
$n \in \bbN \Rightarrow nc <_M a,b$ we have $a < b+c$ and $b < a+c$
\sn
\item[$(c)$]   $E^2_M:a E^2_M b$ \Iff \,  $\bigvee\limits_{n \in \bbN} 
(a <_M b \times_M n \wedge b <_M a \times_M n)$
\sn
\item[$(d)$]   $E^3_M:a E^3_M b$ \Iff \, for some $c \in M$ we have

$\bigwedge\limits_{n \in \bbN}[c^n <_M a \wedge c^n <_M b$ and 
$a < b \times_M c \wedge b < a \times_M c]$
\sn
\item[$(e)$]   $E^4_M:a E^4_M b$ \Iff \, $\bigvee\limits_{n \in \bbN}
(a < b^n \wedge b < a^n)$
\sn
\item[$(f)$]   $E^5_M:a E^5_M b$ \Iff \,  some
order-automorphism of $M$ maps $a$ to $b$
\sn
\item[$(g)$]   $E^6_M:a E^6_M b$ \Iff \, $a,b$ are
equivalent for the minimal convex
equivalent relations on $M$ which $E^5_M$ refines \when
\sn
\item[$(h)$]  an equivalence relation $E$ on a model $M =
  (|M|,<^M,\ldots)$ is convex \when \, $ a <_M b <_M c,a Ec$ implies
  $a Eb$.
\end{enumerate}
\end{definition}

\begin{definition}
\label{a6}
For $\ell \in \{0,1,\dotsc,6\},M$ is called $\ell$-o.r. (for
order-rigid) if for all $a,b \in M,(M_{<a},<_M) \cong (M_{<b},<_M)$
implies that $a E^\ell_M b$.
\end{definition}

\begin{discussion}
While we know that there are rigid linear orders and we know that
there are rigid models of $\PA$, it is harder to build $\ell$-.o.r. models
of $\PA$, our relevant result will be Theorem \ref{d35}.
\end{discussion}

\begin{claim}
\label{a7}  
1) For $\ell \in \{0,\dotsc,6\}$
 the two place relation $E^\ell_M$ is an equivalence relation on 
$M \backslash \bbN$ and is convex except possibly for
 $\ell=5$.

\noindent
1A) Moreover, if $a E^3_M b$ \then \, for every $c \in M$ we have $M
\models (\bigwedge\limits_{n} c^n < a) \Leftrightarrow M \models
(\bigwedge\limits_{n} c^n < b)$ and the set of such $c$'s is closed
under products and sums and is convex.

\noindent
1B)  Also if $a E^1_M b$ \then \, for every $c \in M$
we have $M \models ``(\bigwedge\limits_{n} c \times n < a)
\Leftrightarrow (\bigwedge\limits_{n} (c \times n < b)"$ and the set of
such $c$'s is closed under sums and is convex.

\noindent
2) If $a,b \in M \backslash \Bbb N$ and 
$a E^2_M b$ \then \,  $(M_{<a},<_M) \cong (M_{<b},<_M)$;
moreover there is an automorphism of 
$M \rest \{<\}$ mapping $a$ to $b$, that is, $a E^5_M b$.

\noindent
3) Assume $a,b \in M \backslash \bbN$ \then \, $(M_{<a},<_M) 
\cong (M_{<b},<_M)$ iff $a E^5_M b$.

\noindent
4) $a E^6_M b$ iff there is $c \le_M \text{\rm min}\{a,b\}$ and an
order-automorphism $f$ of $M$ such that {\rm max}$\{a,b\} \le_M f(c)$
\Iff \, this holds for $c = \text{\rm min}\{a,b\}$.
\end{claim}

\begin{remark}
Note that in \S2 we establish a stronger version of \ref{a7}(2), see
\ref{b9} but the proof of
\ref{b9} uses \ref{a7}(2), also the proof of \ref{a7}(2) applies to
weaker versions of PA than the proof of \ref{b9}, see \S4.
\end{remark}

\begin{PROOF}{\ref{a7}}
1) Let $\ell=3$.  If $a_1 E^3_M a_2$ and let $c$ witness it, then $c$ witnesses
also $a_2 E_M^3 a_2$ and $a_2 E_M^3 a_1$, so reflexivity and symmetry
holds.

Lastly, assume $M \models ``a_1 < a_2 < a_3"$; if $a_k E_M^3 a_{k+1}$
and let $c_k$ witness this for $k=1,2$, then the product $c_1 c_2$ witness $a_1
E^3_M a_3$ by part (1A) proved below 
and if $a_1 E^3_M a_3$ then the same witness
gives $a_1 E^3_M a_2 \wedge a_2 E^3_M a_3$ so transitivity and
convexity holds.  

For $\ell=1$ the proof is similar (using part (1B)
instead of part (1A)), also for $\ell=0,2,4$ the proof is even easier
and for $\ell=5,6$ it holds by the definition.

\noindent
1A), 1B)  Check.

\noindent
2) Without loss of generality, assume that $a < (n-1) a<b<na$, where
$1 \le n \in \bbN$, and also there is a $c$ such that
  $(n-1)c=na-b$.  Then $c < a$ because otherwise $(n-1)a \le (n-1)c =
na-b$ hence $b \le a$, contradiction.  Let $X$ be a set of
   representatives for $M/E_0$ and, without loss of generality, assume
   $a,c \in X$.  Now define $f:M \rightarrow M$ by first defining it
   on $X$ and then extending it to all of $M$ in the obvious way.
   (The obvious way is: if $y=x+k$, where $x \in X$ and $k \in \bbZ$,
   then $f(y) = f(x)+k$ and $f$ is the identity on $\bbN$.)  
If $x \in X$, then let

\[
f(x) = \begin{cases} x &\text{ if } x \le c, \\
n(x-a)+b &\text{ otherwise}. \end{cases}
\]

\mn
Clearly, $f(c) = c$ and $f(a)=b$.  Now check that $f$ is as required.

\noindent
3) First, if $f$ exemplifies $a E^5_M b$, i.e. is an automorphism of
   $M$ mapping $a$ to $b$ then $f \rest M_{< a}$ is an isomorphism
   from $(M_{<a},<_M)$ onto $(M_{<b},<_M)$.  Second, if $f$ is an
   isomorphism from $(M_{<a},<_M)$, onto $(M_{<b},<_M)$ we define a
   function $g:M \rightarrow M$ by: $g(c)$ is $f(c)$ if $c <_M a$ and
   is $b + (c-a)$ if $a \le_M c$.  Now check.

\noindent
4) Let $E'_M = \{(a,b)$: for some order-automorphism $f$ of $M$ we have
   $f(\text{min}\{a,b\}) \ge \text{ max}\{a,b\}\}$; clearly this is
   symmetric (by definition), reflexive (using $f =$ the identity) and
   as $f$ is monotonic also convex (i.e. $a \le a_1 \le b_1 \le b 
\wedge a E'_M b \Rightarrow a_1 E'_M b_1$).  
To prove transitivity it is now enough to show $a_1 < a_2 < a_3 
\wedge a_1 E'_M a_2 \wedge a_2 E'_M a_3
\Rightarrow a_1 E'_M a_3$ which hold by composing the automorphisms
$f_1,f_2$ witnessing $a_1 E'_M a_2,a_2 E'_M a_3$ respectively.  
So $E'_M$ is a convex equivalence relation and
obviously $a E^5_M b \Rightarrow a E'_M b$.

Lastly, $E'_M$ is refined by any convex equivalence relation refining
$E^5_M$, so it follows that $E^6_M = E'_M$ so we are done.
\end{PROOF}

\begin{claim}
\label{a9}
1) If $f$ is an order-isomorphism from $M_1$ onto $M_2$ \then \, $f$
   maps $E^0_{M_1}$ onto $E^0_{M_2}$.

\noindent
2) If $f$ is an almost $\{<,+\}$-isomorphism from $M_1$ onto $M_2$
   \then \, $f$ maps $E^1_{M_1}$ onto $E^1_{M_2}$ and $E^2_{M_1}$ onto
   $E^2_{M_2}$.

\noindent
3) Similarly for embeddings (not used).
\end{claim}

\begin{PROOF}{\ref{a9}}
Straight.
\end{PROOF}

\noindent
Recalling Definition \ref{a6}(2)
\begin{theorem}
\label{a13}  
If $M_1$ is 2-o.r. and $M_1 \rest \{<\},M_2 \rest \{<\}$ 
are isomorphic \then \,  $M_1,M_2$ are almost $\{<,+\}$-isomorphic.
\end{theorem}

\begin{remark}
\label{a15}
1) We have not said ``by the same isomorphism".

\noindent
2) The assumption is too strong to be true for models of full $\PA$.  
But it makes sense for weaker versions of PA, 
see \ref{k13} and part of the proof serves
as proof to \ref{b15} so indirectly serves \ref{b13}.
\end{remark}

\begin{question}
\label{a17}  
Are $M_1,M_2$ isomorphic \when \, (the main case is $\ell = 3$):
\mn
\begin{enumerate}
\item[$(a)$]   $M_1,M_2$ are isomorphic as linear orders
\sn
\item[$(b)$]   $M_1$ is $\ell$-o.r. 
\end{enumerate}
\end{question}

\begin{remark}
\label{a19}
In \ref{a17}, it is less desirable but we may consider
adding that also $M_2$ is l.o.r.
\end{remark}

\begin{PROOF}{\ref{a13}}
Easily by \ref{a7}(1)
\mn
\begin{enumerate}
\item[$(*)_0$]   each $E^2_{M_\ell}$-equivalence class is convex.
\end{enumerate}
\mn
Without loss of generality
\mn
\begin{enumerate}
\item[$(*)_1$]   $<_{M_1} = <_{M_2}$ hence $M_1,M_2$ have the same
universe and let 
$M: = M_1 \rest \{<\} = M_2 \rest \{<\}$.
\end{enumerate}
\mn
Also as usual
\mn
\begin{enumerate}
\item[$(*)_2$]   $\bbN \subseteq M_\ell$ for $\ell=1,2$.
\end{enumerate}
\mn
Now
\mn
\begin{enumerate}
\item[$(*)_3$]   if $a,b \in M \backslash \bbN$ and $a <_M b$ and
$a +_{M_2} b = c$ and $a' = c -_{M_1} b$, that is $M_1 \models 
``c = a' + b"$ \then \, $a E^2_{M_1} a'$.
\end{enumerate}
\mn
[Why?  Now $[b,c)_{M_2} = [b,c)_M$ is 
isomorphic to $[0,a)_{M_2}$ as linear orders (as $M_2 \models
``a+b=c"$ as $M_2$ satisfies PA) hence $[b,c)_{M_1}$ is isomorphic to
$[0,a)_{M_1}$ as linear orders.  Of course $[b,c)_{M_1}$ is
isomorphic to $[0,c -_{M_1} b) = [0,a')_{M_1}$ as linear orders.
So $[0,a)_{M_1},[0,a')_{M_1}$ are isomorphic as linear orders.  But $M_1$ is
2-o.r. hence $a E^2_{M_1} a'$ as required].
\mn
\begin{enumerate}
\item[$(*)_4$]   if $a <_M b$ and $b \in M \backslash \bbN$ then $b,a
+_{M_1} b,a +_{M_2} b$ are $E^2_{M_1}$-equivalent.
\end{enumerate}
\mn
[Why?   Similar proof: as by the proof of \ref{a7}(3) as trivially $b
E^2_{M_\ell} (a +_{M_\ell} b)$ for $\ell=1,2$ and $(a +_{M_1} b) 
E^5_{M_1} (a +_{M_2} b)$ so use ``$M_1$ is 2-o.r." to deduce $(a
+_{M_1} b)E^2_{M_1}(a +_{M_2} b)$, together we are done.]
\mn
\begin{enumerate}
\item[$(*)_5$]  if $a,b \in M \backslash \bbN$ and 
$a \times_{M_\ell} b =
c_\ell$ for $\ell=1,2$ \then \, $c_1 E^2_{M_1} c_2$.
\end{enumerate}
\mn
[Why?  For $\ell=1,2$, as 
$M_\ell$ is a model of PA it follows that $(M_{< c_\ell},
<_M)$ is isomorphic to $(M_{<a}$,
$<_M) \times (M_{<b},<_M)$ ordered lexicographically.  
Hence $((M_1)_{< c_1}$,
\newline
$<_M) = (M_{< c_1},<_M)$ and $((M_2)_{< c_2},<_M)
= (M_{< c_2},<_M)$ are isomorphic (and trivially $c_1,c_2 \notin \Bbb N$) 
hence by ``$M_1$ is 2-o.r." we have $c_1 E^2_{M_1} c_2$ as promised.]
\mn
\begin{enumerate}
\item[$(*)_6$]  for $a,b \in M \backslash \bbN$ we have:

$a E^2_{M_2} b$ iff $a E^2_{M_1} b$.
\end{enumerate}
\mn
[Why?  First, assume $a E^2_{M_2} b$; now \wilog \, $a <_M b$, and
say $k \in \bbN,k \ne 0,b < k \times_{M_2} a$ so $b < a +_{M_2} \ldots
+_{M_2} a$ ($k$ summands).  By $(*)_4$ we can prove by induction on
$k$ that $b E^2_{M_1} a$ as required in the ``only if" direction.
(Alternatively by \ref{a7}(2) we have $a E^5_M b$ hence by ``$M_1$ is
2-o.r.", we get $b E^2_{M_1} a$.)

Second, assume $\neg(a E^2_{M_2} b)$ and \wilog \, $a <_M b$; note
that we cannot use the same argument as above interchanging $M_1,M_2$,
because only on $M_1$ we assume its being $2$ - o.r.  As $M_2
\models$ PA there is $c \in M_2$ such that $M_2 \models ``a \times c 
\le b < a \times c + a"$, now $c \notin \bbN$ because we are assuming
$\neg(a E^2_{M_2} b)$.  By $(*)_5$ we have $(a \times_{M_1} c) 
E^2_{M_1}(a \times_{M_2} c)$.  But by $(*)_4$, 
we have $a \times_{M_2} c,a \times_{M_2} c +_{M_2} a$ 
are $E^2_{M_1}$-equivalent hence by the choice of $c$ and 
$(*)_0$ also $a \times_{M_2} c,b$
are $E^2_{M_1}$-equivalent; so together with the previous sentence 
$(a \times_{M_1} c) E^2_{M_1} b$.  But by the definitions, as $c \in M
\backslash \bbN$ clearly $\neg(a E^2_{M_1}(a \times_{M_1} c))$ hence
$\neg(a E^2_{M_1} b)$ as required in the ``if" direction.]
\mn
\begin{enumerate} 
\item[$(*)_7$]   if $a \in M \backslash \bbN$ then
$(a/E^2_{M_1},<_M)$ has cofinality $\aleph_0$ and also its inverse has
cofinality $\aleph_0$.
\end{enumerate}
\mn
[Why?  As $M_1 \models$ PA the sequence $\langle a 
\times_{M_1} 2^n:n \in
\bbN\rangle$ is increasing, the members form an unbounded subset of
$a/E^2_{M_1}$; similarly $\langle\text{min}\{b \in M:a \le b \times_{M_1}
2^n\}:n \in \bbN\rangle$ is decreasing, the members form a
subset of $a/E^2_{M_1}$ unbounded from below, 
recalling the definition of $E^2_{M_1}$.]
\mn
\begin{enumerate}
\item[$(*)_8$]   let $\ell \in \{1,2\}$
if $a \in M \backslash \bbN$ and $X_\ell =
\{(2^b)^{M_\ell}:b \in M_\ell\}$ for $\ell=1,2$ \then \, $X_\ell \cap 
(a/E^2_{M_1})$ has order-type $\bbZ$ and is unbounded in
$(a/E^2_{M_1},<_M)$ from above and from below.
\end{enumerate}
\mn
[Why?  If $a' \in M \backslash \Bbb N$ and 
$\ell \in \{1,2\}$, \then \, as $M_\ell \models$ PA for some $b_\ell$ we have 
$M_\ell \models ``2^{b_\ell} \le a' < 2^{b_\ell +1} = 
2^{b_\ell} + 2^{b_\ell}"$ so recalling the definition of
$E^2_{M_\ell}$ we are done.]

We may hope that: if $a \in M \backslash \Bbb N$ and
$M_\ell \models ``2^a = b_\ell"$ for $\ell=1,2$ \then \, $b_1
E^2_{M_1} b_2$.  Anyhow
\mn
\begin{enumerate}
\item[$\circledast_1$]  for $\ell=1,2$
\sn
\item[${{}}$]  $(a) \quad$ define $f_\ell:M_\ell \rightarrow M_\ell$
by $f_\ell(a) = (2^a)^{M_\ell}$,
\sn
\item[${{}}$]  $(b) \quad$ define: $M^*_\ell$ is the model with universe
$X_\ell := \text{ Rang}(f_\ell)$ such that 

\hskip25pt $f_\ell$ is an 
isomorphism from $M_\ell$ onto $M^*_\ell$ 
\sn
\item[$\circledast_2$]   for $\ell=1,2$ if $a,b \in X_\ell$ 
then $a +_{M^*_\ell} b = a \times_{M_\ell} b$. 
\end{enumerate}
\mn
[Why?  As PA $\vdash 2^x 2^y = 2^{x+y}$.]
\mn
\begin{enumerate}
\item[$\circledast_3$]   if $\ell=1,2$ and $a,b \in M_1 \backslash
\bbN$ \then
\begin{enumerate}
\item[$(a)$]    $a E^0_{M_\ell} b$ iff $f_\ell(a) E^2_{M^*_\ell}
f_\ell(b)$
\sn
\item[$(b)$]  $a E^1_{M_\ell} b$ iff $f_\ell(a)E^3_{M^*_\ell} f_\ell(b)$
\sn
\item[$(c)$]  $a E^2_{M_\ell} b$ iff $f_\ell(a) E^4_{M^*_\ell} f_\ell(b)$.
\end{enumerate}
\end{enumerate}
\mn
[Why?  Look at the definitions and do basic arithmetic.]
\mn
\begin{enumerate}
\item[$\circledast_4$]  there is an order isomorphism $h$ from $X_1$
onto $X_2$ such that
\begin{enumerate}
\item[$(a)$]  $h \rest \{(2^n)^{\bbN}:n \in \bbN\}$ is the identity
\sn
\item[$(b)$]   if $a \in M \backslash \bbN$ then $h$ maps $X_1
\cap (a/E^2_{M_1})$ onto $X_2 \cap (a/E^2_{M_2})$.
\end{enumerate}
\end{enumerate}
\mn
[Why?  By $(*)_8 + (*)_6$.]
\mn
\begin{enumerate}
\item[$\circledast_5$]   if $a,b \in X_1$ then $a E^0_{M^*_1} b
\Leftrightarrow h(a) E^0_{M^*_2} h(b)$.
\end{enumerate}
\mn
[Why?  By $\circledast_3$ and $\circledast_4$ and the 
definition of $E^0_{M^*_\ell}$.]
\mn
\begin{enumerate}
\item[$\circledast_6$]   if $M^*_\ell \models ``a_\ell + b_\ell = c_\ell"$
for $\ell=1,2$ and $h(a_1) = a_2$ and $h(b_1) = b_2$ then 
\begin{enumerate}
\item[$(a)$]   $c_1 E^2_{M_\ell} c_2$ for $\ell=1,2$
\sn
\item[$(b)$]   $c_1 E^0_{M^*_\ell} c_2$ for $\ell=1,2$.
\end{enumerate}
\end{enumerate}
\mn
[Why?  If $a_1 \in \bbN$ or $b_1 \in \bbN$ the conclusion follows easily so
we assume $a_1,b_1 \notin \bbN$. 
For $\ell=1,2$ by $\circledast_2$ we have $M_\ell \models ``a_\ell
\times b_\ell = c_\ell"$.  Also by $\circledast_4(b)$ we have
$x \in X_1 \backslash \bbN \Rightarrow  x E^2_{M_1} h(x)$ 
recalling $E^2_{M_1} = E^2_{M_2}$ by $(*)_6$ we have 
$a_1 E^2_{M_2} a_2,b_1 E^2_{M_2} b_2$ hence for some $n \in \bbN$ we
have $M_2 \models ``a_1 < n \times a_2 \wedge a_2 < n \times a_1
\wedge b_1 < n \times b_2 \wedge b_2 < n \times b_1"$ hence $M_2 \models ``a_1
\times b_1 < n^2 \times a_2 \times b_2 \wedge a_2 \times b_2 < n^2
\times a_1 \times b_1"$ hence $(a_1 \times_{M_2} b_1),(a_2 \times_{M_2} b_2)$
are $E^2_{M_2}$-equivalent and also are $E^2_{M_2}$-equivalent.

So by $(*)_5$ we have $(a_1 \times_{M_2} b_1),(a_1 \times_{M_1} b_1)$
are $E^2_{M_1}$-equivalent and also $(a_2 \times_{M_1} b_2),
(a_2 \times_{M_2} b_2)$ are $E^2_{M_1}$-equivalent
hence together with the previous paragraph 
by $(*)_6$ they are $E^2_{M_\ell}$-equivalent, in
particular $c_1,c_2$ are $E^2_{M_\ell}$-equivalent as required in
clause (a) of $\circledast_6$.  By $\circledast_3 + \circledast_1(b)$  
also clause (b) of $\circledast_6$ there follows.]

So by $\circledast_4 + \circledast_6(b)$ we are done. 
\end{PROOF}

We have used
\begin{observation}
\label{a19}
Assume $a \in M \backslash \bbN$

\noindent
1) $\langle a+n:n \in \bbN \rangle$ is increasing and cofinal in
$a/E^0_M$.

\noindent
2) $\langle a-n:n \in \bbN \rangle$ is decreasing and unbounded from
   below in $a/E^0_M$.

\noindent
3) $\langle n \times a:n \in \bbN \rangle$ is increasing and cofinal in
 $a/E^2_M$.

\noindent
4) Moreover
$\langle\text{min}\{b:n \times_M b \ge a\}:n \in \bbN\}\rangle$ is
decreasing and unbounded from below in $a/E^2_M$.

\noindent
5) Moreover for some $b,2^b \le a <
   2^{b+1}$ hence we can use in (3),(4) the sequence $\langle
   2^{b+n}:n \in \bbN\rangle,\langle 2^{b-n}:n \in \bbN\rangle$. 
\end{observation}

\begin{observation}
\label{a21}
1) Assume $M_k \models ``a \times b =
c_k"$ for $k=1,2$ and $M_1 \rest \{<\} = M_2 \rest \{<\}$.  \Then \, $c_1
E^5_{M_2} c_2$ for $k=1,2$.

\noindent
2) Assume  $M_k \models ``a_1 \times
   a_2 \times \ldots \times a_m = c_k"$ for $k=1,2$ and $M_1 \rest
   \{<\} = M_2 \rest \{<\}$.  \Then \, $c_1 E^5_{M_k} c_2$ for $k=1,2$.
\end{observation}

\begin{PROOF}{\ref{a21}}
1) See $(*)_5$ in the proof of \ref{a13}.

\noindent
2) Similar proof.
\end{PROOF}
\newpage

\section {More for $E^3_M$} 

Here we say more on the equivalence relations $E^\ell_M$.  In
\ref{b3} we deal with basic properties: when $E^\ell_\mu \subseteq
E^{\ell +1}_\mu$, when $\ell$-o.r. implies $(\ell +1)$-o.r.,
preservation under $+$ and $\times$.  We also prove one half of our
answer to \ref{h.1} that is in \ref{b9} we prove $a_1 E^3_M b$ implies
$a_1 E^5_\mu a_2$.  Concerning the weak form of uniqueness of the
additive structure in Theorem \ref{b13} we prove e.g. if $M_1,M_2$ are
order isomorphic and $M_1$ is 3-o.r. then $M_1 \rest \{<,+\},M_2 \rest
\{<,+\}$ are almost isomorphic (i.e. the ``error" in $+$ is finite)
but not necessarily by the same isomorphism.  We end (in \ref{b23})
that $a/E^4_\mu$ is divided by $E^3_\mu$ if $I = \{b/E^3_\mu:b \in
a/E^4_\mu\}$ is naturally ordered, is isomorphic to a subset of
$\bbR$, even one which is an additive subgroup (a ``translation" of
the product in $M$).

\begin{claim}
\label{b3}  
1) $E^\ell_M$ refines $E^{\ell +1}_M$ for $\ell=0,1,2,3,5$.

\noindent
2) If $a_k E^\ell_M b_k$ for $k=1,2$ and $\ell = 0,1,2,3,4,5,6$ 
\then \, $(a_1 + a_2) E^\ell_M (b_1 + b_2)$.

\noindent
3) If $a_k E^\ell_M b_k$ for $k=1,2$ and $\ell = 2,3,4$ \then \, 
$(a_1 \times_M a_2) E^\ell_M(b_1 \times_M b_2)$.

\noindent
4) Part (3) holds also for $\ell=5,6$.

\noindent
5) If $M$ is $\ell$-{\rm o.r.} 
\then \, $M$ is $(\ell +1)$-{\rm o.r.} for $\ell=0,1,2,3$.
\end{claim}

\begin{remark}
\label{b5}
Concerning \ref{b3} recall that $E^3_M$ refines $E^5_M$ by \ref{a7}(3),(4).
\end{remark}

\begin{PROOF}{\ref{b3}}
1) Read the definitions.

\noindent
2) First, assume $\ell =0$, so by the assumption for $k=1,2$ 
there are $m_k,n_k \in \bbN$ such that $M \models ``a_k + m_k = 
b_k + n_k"$.  Now let $m := m_1 + m_2 \in \bbN$ and $n :=
   n_1 + n_2 \in \bbN$ hence $M \models ``(a_1 + a_2) + (m_1 + m_2) =
   (b_1 + b_2) + (n_1 + n_2)"$ hence $(a_1 + a_2) E^0_M (b_1 + b_2)$
   as required.

Second, assume $\ell =2$, so by the assumption, for $k=1,2$ there is
$n_k \in \bbN$ such that $M \models ``a_k < n_k \times b_k \wedge b_k
< n_k \times a_k"$.  Let $n = \text{ max}\{n_1,n_2\} \in \bbN$ hence
$M \models ``(a_1 + a_2) < n_1 b_1 + n_2 b_2 \le nb_1 + nb_2 = 
n(b_1 + b_2)"$ and 
similarly $M \models ``(b_1 + b_2) < n(a_1 + a_2)"$ hence
$(a_1 + a_2) E^2_M (b_1 + b_2)$.

Third, assume $\ell=1,3$;  \wilog \, 
$a_1 <_M b_1$ and as $E_M^\ell$ is convex (see \ref{a7}(1))
\wilog \, $a_2 \le_M b_2$.
Letting $c_k$ witness $a_k E^\ell_M b_k$ for $k=1,2$ easily $c = \text{
max}\{c_1,c_2\}$ witness $(a_1 + a_2) E^\ell_M (b_1 + b_2)$.

Fourth, the case $\ell=4$ is easy, too.

Fifth, assume $\ell=5$ and $f_k$ is an order-automorphism of $M$
mapping $a_k$ to $b_k$ for $k=1,2$.  Define a function $f$ from $M$ to $M$ by
\mn
\begin{enumerate}
\item[$(*)$]  $(a) \quad f(x) = f_1(x)$ if $x <_M a_1$
\sn
\item[${{}}$]  $(b) \quad f(x) = b_1 + f_2(x-a_1)$ if $a_1 \le_M x$.
\end{enumerate}
\mn
Now check.

Sixth, assume $\ell=6$, let $c_k = \text{ min}\{a_k,b_k\}$; by
\ref{a7}(4) there is $d_k \ge \text{ max}\{a_k,b_k\}$ such that
$c_k E^5_M d_k$ so $a_k,b_k \in [c_k,d_k]$ for $k=1,2$.  But $(c_1
+ c_2) E^5_M(d_1 + d_2)$ by the result for $\ell=5$ and $a_1 + b_1,a_2
+ b_2 \in [c_1 + c_2,d_1 + d_2]$ so we are done.

\noindent
3) First, assume $\ell=2$ and for $k=1,2$ let $n_k$ witness $a_k E^2_M b_k$ and
   choose $n=n_1 n_2$ noting that $n_1,n_2 > 0$ by Definition
\ref{a5}(c).  Now $M \models ``a_1 \times a_2 < (n_1 \times
b_1) \times (n_2 \times b_2) = n \times (b_1 \times b_2)"$ and similarly
   $M \models (b_1 \times b_2) < n(a_1 \times a_2)$.

The proof for $\ell=3$ is easy, too.  For $\ell=4$ by the convexity of
$E^4_M$ \wilog \, $a_1 \le a_2,b_1 \le b_2$ and so there are $n,m \in \bbN$
such that $a_2 \le a^n_1,b_2 \le b^m_1$, so $a_1 \times b_1
\le a_2 \times b_2 \le a^n_1 \times b^m_1 \le (a_1 \times b_1)^{n+m}$ hence
$(a_1 \times b_1) E^4_M (a_2 \times b_2)$.

\noindent
4) For $\ell = 5$, as in the proof of \ref{a13}, i.e. if $c_k = a_k
   \times_M b_k$ for $k=1,2$ there is an order isomorphism $h_k$ from
   $M_{< a_k} \times M_{< b_k}$ onto $M_{< c_k}$ and let $f_k$ be an
   order automorphism of $M$ mapping $a_k$ to $b_k$.  Combining there
   is an order-isomorphism $g_1$ from $M_{< c_1}$ onto $M_{< c_2}$ and
   let $g$ be the order automorphism of $M$ such that $g$ extends $g_1$ and
$c_1 \le d \in M \Rightarrow g(d) = c_2 + f_1(d-c_1)$; so $g$
   witness $(a_1 \times b_1) E^5_M (a_2 \times b_2)$ as promised.  

For $\ell=6$ it follows in the proof of part (3).

\noindent
5) By the definition of $m$-o.r. and part (1).
\end{PROOF}

\begin{question}
\label{b7}
Is $E^5_M$ convex for every $M$?
\end{question}

\begin{claim}
\label{b9} 
If $a_1 E^3_M a_2$ \then \, 
there is an order-automorphism of $M$ mapping $a_1$ to $a_2$,
i.e. $a_1 E^5_M a_2$.
\end{claim}

\begin{PROOF}{\ref{b9}}

Without loss of generality $a_1 <_M a_2$.  If $a_1 E^2_M a_2$ then
 $a_1 E^5_M a_2$ by \ref{a7}(2), so
\wilog \, $\neg(a_1 E^2_M a_2)$ hence $n \times a_1 < a_2$ for $n \in \bbN$.
So by the definition of
 $E^3_M$ and the assumption $a_1 E^3_M a_2$, clearly for some $c \in M
 \backslash \bbN$ we have
\mn
\begin{enumerate}
\item[$(*)_1$]   $c <_M a_1,M \models ``c^n < a_1"$ for $n \in \bbN$
and $(c-1) \times_M a_1 <_M a_2 \le_M c \times_M a_1$.
\end{enumerate}
\mn
Clearly $a_2 E^2_M (c \times_M a_1)$ hence again by \ref{a7}(2)
\wilog \,
\mn
\begin{enumerate}
\item[$(*)_2$]   $M \models ``a_2 = c \times_M a_1"$.
\end{enumerate}
\mn
We define an equivalence relation $E$ on $M \backslash \bbN$:
\mn
\begin{enumerate}
\item[$(*)_3$]  $x Ey$ iff $\bigvee\limits_{n} |x-y| < c^n$.
\end{enumerate}
\mn
Clearly $E$ is a convex equivalence relation.
We choose a set $X$ of representatives for $E$ such that $0,a_1,a_2
\in X$, can be done as $0 + c^n = c^n < a_1$ and 
$c^n+ a_1 < 2 \times a_1 < a_2$ for any $n \in \bbN$.

Note
\mn
\begin{enumerate}
\item[$(*)_4$]  if $b_1,b_2 \in M_{\le a_1}$ then $(b_1 E b_2)
\Leftrightarrow (c \times b_1)E(c \times b_2)$.
\end{enumerate}
\mn
[Why?  As we have

\[
|(c \times b_2) - (c \times b_1)| = c \times (|b_2-b_1|)
\]
\mn
so for $n \in \bbN$:

\[
|(c \times b_2) - (c \times b_1)| < c^{n+1} \Leftrightarrow |b_2-b_1|
 < c^n
\]
\mn
so $(*)_4$ is true indeed.]

Now we define a function $f$ from $M$ into $M$ as follows:
\mn
\begin{enumerate}
\item[$(*)_5$]  $(a) \quad$ if $x \in 0/E$, i.e. $\bigvee\limits_{n}
x <_M c^n$ then $f(x)=x$
\sn
\item[${{}}$]  $(b) \quad$ if $y \in X$ and $y \ne 0$ 
then $f(y) = c \times y$
\sn
\item[${{}}$]  $(c) \quad$ if $x \in X \backslash (0/E) 
\wedge x \le_M y \in x/E$ then $f(y) = f(x) + (y-x)$
\sn
\item[${{}}$]  $(d) \quad$ if $x \in X \backslash (0/E) \wedge 
y <_M x \wedge y \in x/E$ then $f(y) = f(x) + (x-y)$
\end{enumerate}
\mn
Note that $f$ is well defined
and is one-to-one order preserving and onto $M$ by $(*)_4$. As $f(a_2)
= c \times a_1 = a_2$ we have $a_1 E a_2$ so we are done.
\end{PROOF}

Comparing with $(*)_7$ of the proof of \ref{a13}
\begin{observation}
\label{b11}
1) For any $a \in M \backslash \bbN$ we have:
\mn
\begin{enumerate}
\item[$(a)$]   the sequence $\langle \lfloor a^{1+2^{-n}} \rfloor:
n \in \bbN\rangle$, that is
$\langle \max\{b:b$ in $M,a$ divides $b$ and 
$(\lfloor b/a \rfloor)^{2^n} \le a\}:n \in \bbN \rangle$
is a decreasing sequence from $\{b: a' < b$ for every $a' \in 
a/E^3_M\}$ unbounded from below in it
\sn
\item[$(b)$]  the sequence $\langle \lfloor a^{1-2^{-n}} \rfloor:
n \in \bbN\rangle$, that is $\langle \max\{b:(\lfloor a/b
\rfloor)^{2^n} \le a\}:n \in \bbN\rangle$
is an increasing sequence included in $\{b:b < a'$ for every 
$a' \in a/E^3_M\}$ and unbounded from above in it.
\end{enumerate}
\mn
2) For $a \in M \backslash \bbN$ we have:
\mn
\begin{enumerate}
\item[$(a)$]   the sequence $\langle \lfloor (1+2^{-n})a \rfloor:
n \in \bbN\rangle$, is a decreasing sequence in $\{b \in M:b$ above
$a/E^1_M\}$ cofinal in it
\sn
\item[$(b)$]  the sequence $\langle \lfloor (1+2^{-n})a \rfloor: n 
\in \bbN\rangle$, is an increasing sequence in $\{b \in M: b$ below
$a/E^1_M\}$ cofinal in it.
\end{enumerate}
\end{observation}

\begin{PROOF}{\ref{b11}}  
Straight.
\end{PROOF}

\begin{theorem}
\label{b13}
1) If $M_1$ is $3$-o.r. and 
$f$ is an order-isomorphism from $M_1$ onto $M_2$ \then
\, $f$ maps $E^k_{M_1}$ onto $E^k_{M_2}$ for $k=3,4$.

\noindent
2) In part (1), moreover $M_1 \rest \{<,+\},M_2 \rest \{<,+\}$ are
almost isomorphic.

\noindent
3) For any $M$ let $E^7_M = \{(a,b):(\lfloor \log_2(a)
\rfloor) E^4_{M_1}(\lfloor \log_2(b) \rfloor)\}$.
Assume there is an order-isomorphism $f$ from $M_1$ onto 
$M_2$ mapping $E^4_{M_1}$ to $E^4_{M_2}$, e.g. as in the conclusion 
of part (1) and $f$ maps $E^7_{M_1}$ onto $E^7_{M_2}$ \then \, $M_1 
\rest \{<,+\},M_2 \rest \{<,+\}$ are almost $\{<,+\}$ isomorphic.
\end{theorem}

\begin{PROOF}{\ref{b13}}
1) By the assumption and by \ref{b9} respectively
\mn
\begin{enumerate}
\item[$(*)_0$]  $(a) \quad E^3_{M_1} \supseteq E^5_{M_1}$
\sn
\item[${{}}$]  $(b) \quad E^3_{M_\ell} \subseteq E^5_{M_\ell}$ for $\ell=1,2$.
\end{enumerate}
\mn
Easily by \ref{a7}(1)
\mn
\begin{enumerate}
\item[$(*)_1$]   each $E^3_{M_\ell}$-equivalence class is convex.
\end{enumerate}
\mn
Without loss of generality
\mn
\begin{enumerate}
\item[$(*)_2$]  $(a) \quad <_{M_1} = <_{M_2}$ hence
\sn
\item[${{}}$]  $(b) \quad M_1,M_2$ have the same universe and
\sn
\item[${{}}$]  $(c) \quad E^5_{M_1} = E^5_{M_2}$ so $E^3_{M_2}
\subseteq E^5_{M_2} \subseteq E^5_{M_1} = E^3_{M_1}$
\sn
\item[${{}}$]  $(d) \quad$ let $M := M_1 \rest \{<\} = M_2 \rest
\{<\}$ hence $E^5_{M_1} = E^5_M = E^5_{M_2} = E^3_{M_1}$.
\end{enumerate}
\mn
Also as usual
\mn
\begin{enumerate}
\item[$(*)_3$]  $\bbN \subseteq M_k$ for $k=1,2$.
\sn
\item[$(*)_4$]  if $a <_M b$ and $b \in M \backslash \Bbb N$ then $b,a
+_{M_1} b,a +_{M_2} b$ are $E^3_{M_1}$-equivalent.
\end{enumerate}
\mn
[Why?  By the definition of $E^3_M$.]
\mn
\begin{enumerate}
\item[$(*)_5$]   if $a \in M,b \in M \backslash \Bbb N$ 
and $a \times_{M_k} b = c_k$ for $k=1,2$ \then \, $c_1 E^3_{M_1}
c_2$ and $c_1 E^5_M c_2$.
\end{enumerate}
\mn
[By \ref{a21}(1) we have $c_1 E^5_M c_2$ and use $(*)_0(a)$ to deduce
$c_1 E^3_{M_1} c_2$.]
\mn
\begin{enumerate}
\item[$(*)_6$]   if $M_\ell \models ``a_1 \times a_2 \times
\ldots \times a_n = b_\ell"$
for $\ell=1,2$ then $b_1 E^3_{M_1} b_2$ and $b_1 E^5_M b_2$.
\end{enumerate}
\mn
[Why?  Similarly to $(*)_5$, i.e. by \ref{a21}(2).]
\mn
\begin{enumerate}
\item[$(*)_7$]   $E^4_{M_1} = E^4_{M_2}$.
\end{enumerate}
\mn
[Why?  Let $a,b \in M \backslash \bbN$ be given.  For $\ell=1,2$ and
$n \in \bbN$ let $a_{\ell,n}$ be
such that $M_\ell \models ``a^n = a_{\ell,n}"$.

First, assume $a E^4_{M_2} b$ and \wilog \, $a < b$.  So for some
$n \in \Bbb N$ we have $M_2 \models ``a < b < a^n"$ so $M_2 \models
``a < b < a_{2,n}"$.  Also $a_{1,n} E^3_{M_1} a_{2,n}$ by $(*)_6$
so by \ref{b3}(1) we have $a_{1,n} E^4_{M_1} a_{2,n}$ hence
for some $m,M_1 \models ``(a_{1,n})^m \ge a_{2,n}"$, in fact, even $m=2$
is O.K.
So $M_1 \models ``b < a_{1,mn}"$ but $a <_M b$, so together $a
E^4_{M_1} b$. 

Second, assume $\neg(a E^4_{M_2} b)$ and \wilog \, $a <_M b$.  So for
every $n \in \bbN,a_{2,n} <_M b$ and by $(*)_6$ we have $a_{2,n}
E^3_{M_1} a_{1,n}$ hence $a_{1,n}/E^3_{M_1}$ has a member $<b$, 
and so in particular $a_{1,n+1}/E^3_{M_1}$
has a member $<b$, but $a_{1,n}/E^3_{M_1}$ is below
$a_{1,n+1}/E^3_{M_1}$ (just think on the definitions) so $a_{1,n} <_M
b$.  As this holds for every $n \in \bbN$ we conclude $\neg(aE^4_{M_1} b)$.]

Let
\mn
\begin{enumerate} 
\item[$(*)_8$]  $I^k_d = \{c \in M:M_\ell \models ``c^n < d"$ for
every $n \in \bbN\}$ for $d \in M$ and $k=1,2$.
\end{enumerate}
\mn

Now
\mn
\begin{enumerate}
\item[$(*)_9$]  $I^1_d = I^2_d$ for $d \in M$.
\end{enumerate}
\mn
[Why?  By $(*)_7$ as $I^\ell_d = \{c:c/E^4_{M_\ell}$ is below $d\}$.]
\mn
\begin{enumerate}
\item[$(*)_{10}$]   $E^3_{M_2} = E^3_{M_1}$.
\end{enumerate}
\mn
[First, if $a_1 E^3_{M_2} a_2$ then by \ref{a7}(3),(4) we have 
$a_1 E^5_{M_2} a_2$
which, recalling $M_1 \rest \{<\} = M = M_2 \rest \{<\}$, 
is equivalent to $a_1 E^5_{M_1} a_2$ which implies (by $M_1$ being
$3$-o.r. which we are assuming) $a_1 E^3_{M_1} a_2$.

Second, assume $\neg(a_1 E^3_{M_2} a_2)$ and \wilog \, $a_1 <_M a_2$. 
As $M_2 \models \text{ PA}$, for some $c \in M$ we have $M_2 \models
``ca_1 \le a_2 < (c+1)a_2"$ so by the previous sentence
 $c \notin I^2_{a_1}$ hence by $(*)_9$ also
$c \notin I^1_{a_1}$.

By $(*)_5$ we have $(c \times_{M_1} a_1) E^3_{M_1} (c \times_{M_2} a_1)$ and as
$c \times_{M_2} a_1 \le_M a_2$ clearly $(c \times_{M_2} a_1)/E^3_{M_1} \le
a_2/E^3_{M_1}$ so together $(c \times _{M_1} a_1)/E^3_{M_1}$ is smaller or
equal to $a_2/E^3_{M_1}$ so for some $a_3 \in a_2/E^3_{M_1}$ we have
$c \times_{M_1} a_1 < a_3$.  As $c \notin I^1_{a_1}$ this implies $a_1,a_3$ are
not $E^3_{M_1}$-equivalent, so by the choice of $a_3$ also $\neg(a_1
E^3_{M_1} a_2)$, so we are done proving $(*)_{10}$.]

Hence by $(*)_7 + (*)_{10}$ part (1) holds.

\noindent
2),3)  By part (1) and \ref{b15} below for the function $x \mapsto
2^{2^x}$, and the equivalence relation $E^4_M$.
\end{PROOF}

\begin{claim}
\label{b15}
The models $M_1,M_2$ are almost $\{<,+\}$-isomorphic \when
\mn
\begin{enumerate}
\item[$(a)$]  $E_k$ is a convex equivalence relation on $M_k
\backslash \bbN$ for $k=1,2$
\sn 
\item[$(b)$]  $h$ is an order-isomorphism from $M_1$ onto $M_2$ mapping
$E_1$ onto $E_2$
\sn
\item[$(c)$]  $f_k$ is a function definable in $M_k$, is increasing,
maps $\bbN$ into $\bbN$, for each $E_k$-equivalence class $Y$ the set
$\{a \in M:f_k(a) \in Y\}$ has the order type of $\bbZ$ and is
unbounded from below and from above in $Y$, for $k=1,2$, of course
\sn
\item[$(d)$]  for $k=1,2$, if $a_1 E_k a_2$ and $b_1 E_k b_2$ then
$(a_1 + b_1)E_k(a_2 + b_2)$.
\end{enumerate}
\end{claim}

\begin{PROOF}{\ref{b15}}
As in the proof of \ref{a13}.
\end{PROOF}

\begin{claim}
\label{b21}
Assume $h$ is an order-isomorphism from $M_1$ onto $M_2$ and $M_1$ is
4-{\rm o.r.}  \Then \, for $a,b \in M$ we have $h(a)E^4_{M_2}h(b)
\Rightarrow a E^4_{M_1} b$.
\end{claim}

\begin{proof}
Without loss of generality $h$ is the identity and let $M_1 \rest
\{<\} = M = M_2 \rest$
\newline
$\{<\}$.  Define $E_M = \{(a,b)$: there are $n \in \bbN,c_1 \in M,
c_2 \in M$ such that $c_1 \le_M \text{ min}_M\{a,b\}$ and 
max$_M\{a,b\} \le c_2$ and $(M_{< c_1})^n \cong M_{< c_2}\}$.  
Easily $E_M$ is a convex equivalence relation,
$E^4_{M_\ell} \subseteq E_M$ for $\ell=1,2$ and $E^4_{M_1} = E_M$.
\end{proof}

\begin{claim}
\label{b23}
1) Assume $a \in M$ is non-standard and 
$I = \{b/E^3_M:b \in a/E^4_M\}$, naturally
ordered.  \Then \, the linear order $I$ can be embedded into
$\bbR_{> 0}$ with dense image.

\noindent
2) If $M$ is $\aleph_1$-saturated then the embedding is onto $\bbR_{>0}$.

\noindent
3) Moreover defining $+_I$ by $(b_1/E^3_M) +_I (b_2/E^3_M) =
   (b_3/E^3_M)$ when $b_1 \times_M b_2 = b_3$, the embedding commute
   with addition so the image is an additive sub-semi-group
 of $\bbR$.  Also $1_{\bbR}$ belongs to the image.
\end{claim}

\begin{PROOF}{\ref{b25}}
As in Definition \ref{d1} but we elaborate.

Fix $a \in M$ and for $b \in a/E^4_M$ let $\cS_{a,b} =
\{\frac{m_1}{m_2}:m_1,m_2 \in \bbN \backslash \{0\}$ and $M \models
``b^{m_2} \ge a^{m_1}"\}$.  Clearly $S_b$ is a subset of $\bbQ_{>0}$ and
as $M$ is a model of $\PA$ clearly $S_b$ is an initial segment of
$\bbQ_{>0}$.  By the definition of ``$b \in a/E^4_M$ necessarily $S_b
\ne \emptyset$ and $S_b \ne \bbQ_{>0}$, so together $r_b = \sup(\cS_b)$ belongs
to $\bbR_{>0}$.  

Again by $\PA$
\mn
\begin{enumerate}
\item[$(a)$]  $r_{b_1} \le r_{b_2}$ when $b_1 \le_M b_2$ are
  from $a/E^4_M$.
\end{enumerate}
\mn
[Why?  See the definition of $\cS_{b_1},S_{b_2}$.]
\mn
\begin{enumerate}
\item[$(b)$]  $r_{b_1} = r_{b_1} \Leftrightarrow b_1 E^3_M b_2$ for
  $b_1,b_2 \in a/E^4_M$.
\end{enumerate}
\mn
[Why?  By the definition of $E^3_M$.]
\mn
\begin{enumerate}
\item[$(c)$]  if $\bbQ \models ``\frac{m_1}{m_2} < \frac{m_3}{m_4}"$
  where $m_\ell \in \bbN \backslash \{0\}$ for $\ell=1,2,3,4$ \then \, 
for some $b \in a/E^4_M$ we have $\bbR \models ``\frac{m_1}{m_2} \le r_b <
  \frac{m_3}{m_4}"$. 
\end{enumerate}
\mn
[Why?  Let $n \in \bbN \backslash \{0\}$ be such that $M \models ``b <
a^n"$, exists as $b \in a/E^4_M$.  \Wilog \, $m_2 = m_4$ cal it $m$, so
necessarily $m_1 < m_3$ and \wilog \, $m_1 + n < m_3$.  Now by the
definition of $b \mapsto r_b$ the demand on $b$ means that $M \models
``b^m \ge a^{m_1}$ and $b^m < a^{m_3}"$.  Let $b$ be the minimal
member of $M$ such that $M \models ``b^m \ge a^{m_1}"$ hence $M
\models ``b^{m-1} < a^{m_1}$ hence $b^m < a^{m_1} b \le a^{m_1} a^n =
a^{m_1+n} \le a^{m_3}"$ so $b$ is as required.]
\mn
\begin{enumerate}
\item[$(d)$]  $\{r_b:b \in a/E^4_M\}$ is a dense subset of $\bbR_{>0}$.
\end{enumerate}
\mn
[Why?  By (c).]
\mn
\begin{enumerate}
\item[$(e)$]  If $M$ is $\aleph_1$-saturated then $\{r_b:b \in
  a/E^4_M\} = \bbR_{>0}$, i.e. part (2).
\end{enumerate}
\mn
[Why?  For any real $r$ and $n$ we can find $b_1,b_2 \in a/E^4_M$ such
that $r - \frac 1n < r_{b_2} < r < r_{b_2} < \frac 1n$ by (d) and
``$\bbQ$ is dense in $\bbR$".]
\mn
\begin{enumerate}
\item[$(f)$]  part (3) of the claim holds.
\end{enumerate}
\mn
[Why?  By the rules of exponentiation which can be phrased in PA.]

Together we are done.
\end{PROOF}

\begin{remark}
\label{b25}
So the ``distance" between $E^3_M$ and $E^4_M$ is small.
\end{remark}
\newpage

\section {Constructing somewhat rigid models}

\begin{hypothesis}
\label{d0}
$\lambda$ is regular.
\end{hypothesis}

\begin{definition}
\label{d1}
For any $M$ (model of PA) 
\mn
\begin{enumerate}
\item[$(a)$]  let $\bbZ_M = \bbZ[M]$ be the ring $M$ generates 
(so $a \in \bbZ_M$
iff $a = b \vee a = -b$ for some $b \in M$, of course $\bbZ_M$ is
determined only up to isomorphism over $M$; similarly below); when, as
usual, $M$ is ordinary \wilog \, $\bbZ_M \supseteq \bbZ$
\sn
\item[$(b)$]  let $\bbQ_M = \bbQ[M]$ be the field of quotients of $\bbZ_M$; in
fact, it is an ordered field, if $M$ is ordinary then \wilog \, 
$\bbQ_M \supseteq \bbQ$
\sn
\item[$(c)$]  let $\bbR_M = \bbR[M]$ be the closure of $\bbQ_M$ adding all
definable cuts, so in particular it is a real closed field, see
\ref{d1f} below 
\sn
\item[$(d)$]  let $S_M = \bbR^{\text{bd}}_M/\bbR^{\text{infi}}_M$
where (bd stands for bounded, infi stands for infitesimal)
\sn
\begin{enumerate}
\item[$\bullet$]  $\bbR^{\text{bd}}_M =
\{a \in \bbR_M:\bbR_M \models -n < a < n$ for some $n \in \bbN\}$
\sn
\item[$\bullet$]  $\bbR^{\text{infi}}_M = \{a \in
\bbR^{\text{bd}}_M:\bbR_M \models ``-1/n < a < 1/n"$ for every $n \in
\bbN\}$
\sn
\item[$\bullet$]  $\bold j_M$ is the function from
$\bbR^{\text{bd}}_M$ into $\bbR$ such that $M \models ``n_1/m_1 < a <
n_2/m_2" \Rightarrow \bbR \models ``n_1/m_1 < \bold j_M(a) < n_2/m_2$
for $n_1,n_2,m_1,m_2 \in \bbZ"$ such that $m_1,m_2 > 0$.
\end{enumerate}
\end{enumerate}
\end{definition}

\begin{remark}
\label{d1f}
1) On \ref{d1}(d) see Claim \ref{b23}.

\noindent
2) Concerning \ref{d1}(c) note that $\bbR[M]$ is a sub-field of the
   Scott-Cauchy completion $\bbR[M]$ of $\bbQ[M]$ and that for so
   called ``rather classless" models $M,\bbR[M]$ coincide with
   $\bar{\bbR}[M]$.  See more on completions in \cite{Sh:757}.
\end{remark}

\begin{definition}
\label{d2}
Let AP = AP$_\lambda$ be the set of $\bold a$ such that
\mn
\begin{enumerate}
\item[$(a)$]  $\bold a = (M,\Gamma) = (M_{\bold a},\Gamma_{\bold a})$
but we may\footnote{alternatively, replace $M^0_b$ by 
$(M_{\bold a})_{< b}$ in the proof of \ref{d31} below.} write
 $M^{\bold a}_{< b}$ instead $(M_{\bold a})_{< b}$
\sn
\item[$(b)$]  $M$ is a model of PA
\sn
\item[$(c)$]  $|M|$, the universe of $M$, is an ordinal $< \lambda^+$
\sn
\item[$(d)$]  $\Gamma$ is a set of $\le \lambda$ of types over $M$
\sn
\item[$(e)$]  each $p \in \Gamma$ has the form $\{a_{p,\alpha} < x <
b_{p,\alpha}:\alpha < \lambda\}$ where $\alpha < \beta \Rightarrow M
\models a_{p,\alpha} < a_{p,\beta} < b_{p,\beta} < b_{p,\alpha}$
\sn
\item[$(f)$]  $M$ omits every $p \in \Gamma$.
\end{enumerate}
\end{definition}

\begin{definition}
\label{d4}
1) $\le_{\AP}$ is the following two-place relation on AP:

$\bold a \le_{\AP} \bold b$ \Iff \, $M_{\bold a} \prec M_{\bold
b}$ and $\Gamma_{\bold a} \subseteq \Gamma_{\bold b}$.

\noindent
2) Let AP$_T = \{\bold a \in \text{ AP}:M_{\bold a}$ is a model of
   $T\}$.

\noindent
3) Let AP$^{\sat} =\{\bold a \in \AP:M_{\bold a}$ is
   saturated$\}$ and AP$^{\sat}_T = \AP_T \cap \AP^{\sat}$. 
\end{definition}

\begin{claim}
\label{d5}
1) $\le_{\AP}$ is a partial order of {\rm AP}.

\noindent
2) If $\langle \bold a_\alpha:\alpha < \delta\rangle$ is 
   $\le_{\AP}$-increasing, $\delta$ a limit ordinal 
$< \lambda^+$ \then \, $\bold a_\delta = \cup\{\bold a_\alpha:\alpha <
   \delta\}$ defined by $M_{\bold a_\delta} = \cup\{M_{\bold
   a_\alpha}:\alpha < \delta\},\Gamma_{\bold a_\delta} =
 \cup\{\Gamma_{\bold a_\alpha}:\alpha < \delta\}$, is a
   $\le_{\AP}$-{\rm lub} of $\langle \bold a_\alpha:\alpha < \delta\rangle$.

\noindent
3) Assume $\lambda = \lambda^{< \lambda} > \aleph_0$.  
If $\bold a \in \AP$ \then \, there is $\bold b$ such that
$\bold a \le_{\AP} \bold b$ and $M_{\bold b}$ is saturated (of
cardinality $\lambda$).
\end{claim}

\begin{PROOF}{\ref{d5}}
Easy.
\end{PROOF}

\begin{mc}
\label{d31}
1) If $(A)$ then $(B)$ where:
\mn
\begin{enumerate}
\item[$(A)$]  $(a) \quad \lambda = \aleph_0,\bold a \in \AP$
\sn
\item[${{}}$]  $(b) \quad M_{\bold a} \models ``a_* > 
b_* > n"$ for $n \in \bbN$ and $a_*,b_*$ are not
$E^3_{M_{\bold a}}$-equivalent
\sn
\item[${{}}$]  $(c) \quad F$ is an order isomorphism from 
$M^{\bold a}_{< a_*}$ onto $M^{\bold a}_{<b_*}$
\sn
\item[$(B)$]   there are $\bold b,c_*$ satisfying
\begin{enumerate}
\item[$(a)$]  $\bold a \le_{\AP} \bold b$
\sn
\item[$(b)$]  $c_* <_{M_{\bold b}} a_*$ so $c_* \in M_{\bold b}$ but
$c_* \notin M_{\bold a}$
\sn
\item[$(c)$]   some $p \in \Gamma_{\bold b}$ is equivalent to
$\{F(a_1) < x < F(a_2):a_1,a_2 \in M_{\bold a}$ and
$M_{\bold b} \models a_1 < c_* < a_2 \le a_*\}$
  recalling $F$ is the isomorphism from (A)(c).
\end{enumerate}
\end{enumerate}
\end{mc}

\begin{remark}
\label{d32}
1) We use $a_*,b_*,c_*$ as they are constant during the proof of
   \ref{d31}, and we like to let $a,a_i$, etc. be free to denote other
   things.

\noindent
2) Below in the Discussion \ref{d34} we try to explain the proof of
   \ref{d31}; of course, it cannot be fully digested per se,
   \underline{but} the reader can try to look at it from time to time, 
particularly when you lose track in the proof, hopefully this will help.
\end{remark}

\begin{PROOF}{\ref{d31}}

\noindent
\underline{Stage A}:

Let 
\mn
\begin{enumerate}
\item[$\boxplus_1$]   $(a) \quad \Phi = \Phi_{\bold a}$ is the set of
formulas $\varphi(x) = \varphi(x,\bar a)$ with $\varphi(x,\bar y) \in
\bbL(\tau_{\text{PA}})$ 

\hskip25pt and $\bar a \in {}^{\ell g(\bar y)}(M_{\bold a})$
\sn
\item[${{}}$]  $(b) \quad \varphi'(x) \vdash \varphi''(x)$ means
both are from $\Phi$ and 

\hskip25pt $M_{\bold a} \models ``(\forall x)(\varphi'(x)
\rightarrow \varphi''(x)"$
\sn
\item[${{}}$]  $(c) \quad \varphi(M_{\bold a})= \varphi(M_{\bold
 a},\bar a)  = \{b \in M_{\bold a}:
M_{\bold a} \models \varphi[b,\bar a]\}$ if $\varphi = \varphi(x) =$

\hskip25pt $\varphi(x,\bar a) \in \Phi$
\sn
\item[${{}}$]  $(d) \quad$ we define $\Sigma^k_{\bold a} = 
\Sigma^k_{M_{\bold a}}$ as the set of

\hskip25pt $\sigma(x_0,\dotsc,x_{k-1})= \sigma(x_0,\dotsc,x_{k-1},\bar
a)$ satisfying

\hskip25pt $\sigma(\bar x,\bar y)$ is a definable function
in $M_{\bold a}$ and $\bar a \in {}^{\ell g(\bar y)}(M_{\bold a})$

\hskip25pt and $k \in \bbN$, we may omit $k$ when it is 1 and so may
write $\sigma(x,\bar y)$ 

\hskip25pt and $\sigma(x)$
\sn
\item[${{}}$]  $(e) \quad$ let\footnote{the $|\varphi(M_{\bold a})|$
and log$_2$ has natural meanings; of course we can translate it to a
formula in $\bbL(\tau_{\PA})$.}
$\xi(\varphi) = \xi(\varphi(x)) = \bold
j_M(\log_2(|\varphi(M_{\bold a})|)/\log(a_*))$ for $\varphi = \varphi(x) \in 
\Phi_{\bold a}$ 

\hskip25pt such that $\varphi(M) \subseteq [0,a_*)_M$, see \ref{d1}(d)
\sn
\item[${{}}$]  $(f) \quad$ if $\varphi_1,\varphi_2 \in
  \Phi,\varphi_2(M) \ne \emptyset$ and $\sigma \in \Sigma_{\bold a}$
  \then\footnote{So $\sigma'$ is $\sigma$ restricting the domain to
  $\varphi_1(M)$ and rounding the image to be in $\varphi_2(M)
  \cup \{0\}$.}  \, let $\sigma' := \sigma[\varphi_1,\varphi_2]$ be

\hskip25pt  the following function from $\varphi_1(M)$ to 
$\varphi_2(M) \cup \{0\}$, definable in 

\hskip25pt $M:M \models \sigma'(a)=b$ iff $a \in \varphi_1(M)$ and
  $b = \max\{b$: either $b=0$ 

\hskip25pt  and $\sigma(a) < \min(\varphi_2(M))$
 \underline{or} $b \in \varphi_2(M)$ and $b \le \sigma(a)\}$.
\end{enumerate}
\mn
We define $\bbP$, it will serve as a set of approximations to 
tp$(c_*,M_{\bold a},M_{\bold b})$, as follows:
\mn
\begin{enumerate}
\item[$\boxplus_2$]   the quasi-order $\bbP$ is defined by:
\begin{enumerate}
\item[$(a)$]  a member of $\bbP$ is a pair $\bar\varphi = 
(\varphi_1,\varphi_2)$ such that:
\sn
\item[${{}}$]  $(\alpha) \quad \varphi_\ell = \varphi_\ell(x)$ are
from $\Phi$
\sn
\item[${{}}$]  $(\beta) \quad \varphi_1(x) \vdash x < a_*$ and $\varphi_2(x)
\vdash x < b_*$
\sn
\item[${{}}$]  $(\gamma) \quad$ if $a_1 < a_2$ are from
$\varphi_1(M_{\bold a})$ then $[F(a_1),F(a_2)]_M \cap \varphi_2(M) \ne
\emptyset$
\sn
\item[${{}}$]  $(\delta) \quad \xi(\varphi_1(M)) > \xi(\varphi_2(M))$
\sn
\item[$(b)$]  $\bbP \models \bar\varphi' \le \bar\varphi''$ \Iff \,
$\varphi''_\ell(x) \vdash \varphi'_\ell(x)$,  for $\ell = 1,2$.
\end{enumerate}
\end{enumerate}
\mn
Note that $(\varphi_1(x,\bar a),\varphi_2(x,\bar a_2)) \in \bbP$ is
not definable in $M_{\bold a}$ mainly because of $(a)(\gamma)$ of $\boxplus_2$.

Obviously observe
\mn
\begin{enumerate}
\item[$\boxplus_3$]  if $\bar\varphi \in \bbP$ and $\varphi(x) \in
\Phi$ then for some $\bold t \in \{0,1\}$ we have $(\varphi_1(x) \wedge
\varphi(x)^{\bold t},\varphi_2(x)) \in \bbP$ (and is
$\le_{\bbP}$-above $\bar\varphi$; recall that $\varphi^{\bold t}$ is
$\varphi$ is $\bold t=1$ and is $\neg \varphi$ if $\bold t=0$).
\end{enumerate}
\mn
[Why?  $M_{\bold a} \models ``|\varphi_1(M) \cap \varphi(M)| \ge
\frac{1}{2}|\varphi_1(M)|$ or $M_{\bold a} \models ``|\varphi_1(M)
\backslash \varphi(M)| \ge \frac{1}{2}|\varphi_2(M)|"$.

As $a_* \notin \bbN$ clearly $\xi(\varphi_1(x) \wedge \varphi(x)) =
\xi(\varphi_2(x))$ or $\xi(\varphi_2(x) \wedge \neg \varphi(x)) =
\xi(\varphi_2(x))$.  So clearly we are done proving $\boxplus_3$.]
\bigskip

\noindent
\underline{Stage B}:

We arrive at a major point: how to continue to omit members of
$\Gamma_{\bold a}$
\mn
\begin{enumerate}
\item[$\boxplus_4$]  if $\bar\varphi \in \bbP,\sigma(x) \in
\Sigma_{\bold a}$ and $p(x) \in \Gamma_{\bold a}$ \then \, for some
$\bar\varphi'$ and $n$ we have
\begin{enumerate}
\item[$(a)$]  $\bar\varphi \le \bar\varphi' \in \bbP$
\sn
\item[$(b)$]  $\varphi'_0(x) \vdash ``\sigma(x) \notin
(a_{p,n},b_{p,n})"$.
\end{enumerate}
\end{enumerate}
\mn
The rest of this stage is dedicated to proving this.  We use a ``wedge
question". 
\bigskip

\noindent
\underline{Case 1}:  There is $d \in M_{\bold a}$ such that
$\bar\varphi' := (\varphi_1(x) \wedge \sigma(x) = d,\varphi_2(x)) \in
\bbP$.

In this case obviously $\bar\varphi \le \bar\varphi' \in \bbP$.  Also
as $d \in M_{\bold a}$ and $M_{\bold a}$ omit $p(x)$ recalling $p(x)
\in \Gamma_{\bold a}$, clearly $d$ does not realize $p(x)$ hence for
some $n,d \notin (a_{p,n},b_{p,n})_{M_{\bold a}}$; so $\bar\varphi',n$
are as promised.
\bigskip

\noindent
\underline{Case 2}:  Not case 1.

So
\mn
\begin{enumerate}
\item[$(*)_{4.1}$]   $\xi(\varphi(x) \wedge \sigma(x)=d) \le
\xi(\varphi_2(x))$ for every $d$ from $M_{\bold a}$.
\end{enumerate}
\mn
Clearly there is a minimal $d_* \in M_{\bold a}$ satisfying
\mn
\begin{enumerate}
\item[$(*)_{4.1}$]   $M_{\bold a} \models ``|\{c \in
\varphi_1(M):\sigma(c) \le d_*\}| \ge \frac{1}{2}|\varphi_1(M_{\bold a})|"$.
\end{enumerate}
\mn
So
\mn
\begin{enumerate}
\item[$(*)_{4.2}$]   $M_{\bold a} \models ``|\{c \in
\varphi_1(M):\sigma(c) \ge d_*\}| \ge \frac{1}{2}|\varphi_2(M_{\bold a})|"$.
\end{enumerate}
\mn
But $M_{\bold a}$ omits the type $p(x)$ as $p(x) \in \Gamma_{\bold a}$
and $d_* \in M_{\bold a}$, so for some $n,d_* \notin
(a_{p,n},b_{p,n})$.

So one of the following sub-cases occurs:

\noindent
\underline{Sub-case 2a}:  $d_* \le a_{p,n}$.

Let $\varphi'_1(x) = \varphi_1(x) \wedge (\sigma(x) \le a_{p,n+1})$ and
$\varphi'_2(x) = \varphi_2(x)$.  Now the pair $\bar\varphi' =
(\varphi'_1,\varphi'_2)$ is as required, noting (by $(*)_{4.1}$) that

\[
M_{\bold a} \models ``|\varphi'_1(M)| \ge |\{a \in
\varphi_1(M):\sigma(c) \le d_*\}| \ge \frac 12 |\varphi_1(M)|"
\]

\mn
hence 

\[
\xi(\varphi'_1(x)) = \xi(\varphi_1(x)) > \xi(\varphi_2(x)) =
\xi(\varphi'_2(x)). 
\]

\bigskip

\noindent
\underline{Sub-case 2b}:  $d_* \ge a_{p,n}$.

Let $\varphi'_1(x) = \varphi_2(x) \wedge (\sigma(x) \ge b_{p,n+1})$ and
$\varphi'_2(x) = \varphi_2(x)$.  Now $\bar\varphi' =
(\varphi'_1,\varphi'_2)$ is as required noting (by $(*)_{4.2}$) that

\[
M_a \models ``|\varphi'_1(M)| \ge |\{c \in \varphi_1(M):
\sigma(c) \ge d_*\}| \ge \frac 12 |\varphi_1(M)|"
\]

\mn
hence 

\[
\xi(\varphi'_1(x)) = \xi(\varphi_1(x)) > \xi(\varphi_2(x)) =
\xi(\varphi'_2(x)). 
\]

\mn
So we are done proving $\boxplus_4$. 
\bigskip

\noindent
\underline{Stage C}:

How do we omit the new type?  Recall that $c_*$ will realize a type to which
$\bar\varphi \in \bbP$ is an approximation and we have to omit the
relevant type from clause $(B)(c)$ of the Claim.

This stage is dedicated to proving the
relevant statement:
\mn
\begin{enumerate}
\item[$\boxplus_5$]  if $\bar\varphi \in \bbP$ and $\sigma(x) \in
\Sigma_{\bold a}$ \then \, for some $\bar\varphi'$:
\begin{enumerate}
\item[$(a)$]  $\bar\varphi \le \bar\varphi' \in \bbP$
\sn
\item[$(b)$]  for some $a_1 < a_2 \le a_*$ we have
\sn
\item[${{}}$]  $\bullet \quad \varphi'_1(x) \vdash a_1 \le x < a_2$
\sn
\item[${{}}$]  $\bullet \quad \varphi'_1(x) \vdash \neg(F(a_1) \le
\sigma(x) < F(a_2))$.
\end{enumerate}
\end{enumerate}
\mn
First note
\mn
\begin{enumerate}
\item[$(*)_{5.1}$]  if there is $\bar\varphi' \in \Phi$ such that
$\bar\varphi \le \bar\varphi'$ and $\xi(\bar\varphi'_1)/\xi(\bar\varphi'_2) >
2$ \then \, the conclusion of $\boxplus_5$ holds.
\end{enumerate}
\mn
[Why?  As $|\bigcup\limits_{i} A_i| = \sum\limits_{i} |A_i|$ for
pairwise disjoint sets, i.e. the version provable in PA  
we can find $\varphi''_1(x) \in
\Phi$ such that $\varphi''_1(M_{\bold a}) \subseteq
\varphi'_1(M_{\bold a}),M_{\bold a} \models ``|\varphi'_1(M_{\bold a})|/
|\varphi''_1(M_{\bold a})| \le |\varphi_2(M_{\bold a})|"$ and
$\sigma[\varphi''_1,\varphi_2]$, which was defined in Clause
$\boxplus_1(f)$, is constant 
say constantly $e$, hence $e \in \varphi_2(M_{\bold a})$.  So
$\xi(\varphi''_1(x)) \ge \xi(\varphi'_1(x)) - \xi(\varphi'_2(x)) > 2
\xi(\varphi'_2(x)) - \xi(\varphi'_2(x)) = \xi(\varphi'_2(x))$ hence
$(\varphi''_1,\varphi'_2)$ belongs to $\bbP$ and $\bar\varphi \le
\bar\varphi' \le (\varphi''_1,\varphi'_2)$.
As $e \in \varphi_2(M_{\bold a}) \subseteq [0,b_*)_{M_{\bold a}}$ and
$F$ is onto $M^{\bold a}_{< b_*}$ for some $d <_{M_{\bold a}} a_*$ we
have $F(d) = e$.  By $\boxplus_3$ \wilog \, $\varphi''(x) \vdash ``x <
d"$ or $\varphi''(x) \vdash ``d \le x"$ so we can choose $(a_1,a_2)$ as
$(0,d)$ or as $(d,a_*)$.  So we are done.]

So we can assume $(*)_{5.1}$ does not apply.  Hence it is natural to
deduce (can replace $\frac{1}{8}$ by any fixed $\varepsilon > 0$).
\mn
\begin{enumerate}
\item[$(*)_{5.2}$]  \Wilog \, for no $\bar\varphi' \in \Phi$ do we
have $\bar\varphi \le \bar\varphi'$ and

\[
\xi(\varphi'_1)/\xi(\varphi'_2) > (1 +
\frac{1}{8})\xi(\varphi_1)/\xi(\varphi_2).
\]
\end{enumerate}
\mn
[Why?  We try to choose $\bar\varphi^n$ by induction on $n \in \bbN$
such that $\bar\varphi^n \in \bbP,\bar\varphi^0 =
\bar\varphi,\bar\varphi^n \le \bar\varphi^{n+1}$ and
$\xi(\varphi^n_1)/\xi(\varphi'_2) \ge (1 + \frac{1}{8})^n
\xi(\varphi_1)/\xi(\varphi_2)$.  So for some $n$ we have
$\xi(\varphi^n_1)/\xi(\varphi'_2) > 2$ and we can apply $(*)_{5.1}$,
contradiction.  But $\bar\varphi^0$ is well defined, hence for some
$n,\bar\varphi^n$ is well defined but we cannot choose
$\bar\varphi^{n+1}$.  Now $\bar\varphi^n$ is as required in
$(*)_{5.2}$.]  

Clearly
\mn
\begin{enumerate}
\item[$(*)_{5.3}$]  if $a_1 < a_2$ are from $\varphi_1(M)$, so $b_1 \le
F(a_1) < F(a_2) \le b_2$ are from $\varphi_2(M)$ and $\xi(\varphi_1(x)
\wedge a_1 \le x < a_2 \wedge \sigma(x) \notin [b_1,b_2)) >
\xi(\varphi_2(x) \wedge b_1 \le x < b_2)$ \then \, we are done.
\end{enumerate}

So from now on we assume that there are no $a_1,a_2$ as in
$(*)_{5.3}$.

Let
\mn
\begin{enumerate}
\item[$(*)_{5.4}$]  $(a) \quad k_* \in \bbN \backslash \{0\}$ be large
enough such that $(\xi(\varphi_1)-\xi(\varphi_2))/\xi(\varphi_2) > 2/k_*$
\sn
\item[${{}}$]  $(b) \quad$ let $n(1)\in \bbN$ be large enough such that:
\begin{enumerate}
\item[${{}}$]  $\bullet \quad \xi(\varphi_1)-\xi(\varphi_2) > 1/n(1)$ 
\sn
\item[${{}}$]  $\bullet \quad \xi(\varphi_2) > (k_*+1)/n(1)$
\end{enumerate}
\item[${{}}$]  $(c) \quad$ let $n(2)\in \bbN$ be $> n(1)$
\sn
\item[$(*)_{5.5}$]  let
\sn
\item[${{}}$]  $(a) \quad \psi_1(x_1,x_2;y_1,y_2) = x_1 < x_2 \wedge
\varphi_1(x_1) \wedge \varphi_1(x_2) \wedge y_1 < y_2 < b_*$
\sn
\item[${{}}$]  $(b) \quad \psi_2(x_1,x_2;y_1,y_2)$ is the conjunction
of:
\begin{enumerate}
\item[${{}}$]  $\bullet \quad \psi_1(x_1,x_2,y_1,y_2)$
\sn
\item[${{}}$]  $\bullet \quad |\{x:\varphi_1(x) \wedge x_1 \le x <
x_2\}|^{n(1)} \ge |\{y:\varphi_2(y) \wedge y_1 \le y < y_2\}|^{n(1)}
\times a_*$
\end{enumerate}
\item[${{}}$]  $(c) \quad \vartheta_2(x_1,x_2) = (\exists
y_1,y_2)(\psi_2(x_1,x_2;y_1,y_2)$
\sn
\item[${{}}$]  $(d) \quad \psi_3(x_1,x_2;y_1,y_2)$ is the conjunction of
\begin{enumerate}
\item[${{}}$]  $\bullet \quad \psi_2(x_1,x_2;y_1,y_2)$
\sn
\item[${{}}$]  $\bullet \quad |\{x:\varphi_1(x) \wedge x_1 \le x < x_2 \wedge
\sigma(x) \notin [y_1,y_2)\}|^{n(2)} < |\{y:\varphi_2(a) \wedge y_1$

\hskip25pt $\le y < y_2\}|^{n(2)} \times a_*$
\end{enumerate}
\item[${{}}$]  $(e) \quad \vartheta_3(x_1,x_2) = (\exists
  y_1,y_2)\psi_3(x_1,x_2,y_1,y_2)$.
\end{enumerate}
\mn
So by our assumptions (for clause (b) use ``$(*)_{5.3}$ does not
apply") we have
\mn
\begin{enumerate}
\item[$(*)_{5.6}$]  $(a) \quad$ if $a_1 < a_2$ are from
$\varphi_1(M_{\bold a})$ then $M_{\bold a} \models
\psi_1[a_1,a_2;F^{[\varphi_2]}(a_1),F^{[\varphi_2]}(a_2)]$ 
\sn
\item[${{}}$]  $(b) \quad$ if $M_{\bold a} \models
\psi_2[a_1,a_2;F^{[\varphi_2]}(a_1),F^{[\varphi_2]}(a_2)]$ then
\sn
\begin{enumerate}
\item[$\bullet$]  $M_{\bold a} \models \psi_3[a_1,a_2;F^{[\varphi_2]}(a_1),
F^{[\varphi_2]}(a_2)]$
\sn
\item[$\bullet$]  $M_{\bold a} \models \vartheta_3[a_1,a_2]$.
\end{enumerate}
\end{enumerate}
\mn
Clearly
\mn
\begin{enumerate}
\item[$(*)_{5.7}$]   if $M_{\bold a} \models
\psi_3[a_1,a_2;b^\iota_1,b^\iota_2]$ for $\iota=1,2$ then
$[b^1_1,b^1_2)_{M_{\bold a}} \cap [b^2_1,b^2_2)_{M_{\bold a}} \ne
\emptyset$.
\end{enumerate}
\mn
It is well known that for a linear order, for any finite family of
intervals, their intersection is non-empty iff the intersection of
any two is non-empty.  

Now a version of this can be proved in PA hence
\mn
\begin{enumerate}
\item[$(*)_{5.8}$]   for some $\sigma(x_1,x_2) \in \Sigma^2_{\bold a}$,
we have: if $M_{\bold a} \models \vartheta_3[a_1,a_2]$ \then \,
$\sigma(a_1,a_2) \in \varphi_2(M)$ and for every $b_1,b_2 \in
\varphi_2(M)$ we have $M_{\bold a} \models \psi_3[a_1,a_2;b_1,b_2]$
implies $\sigma(a_1,a_2) \in [b_1,b_2)_{M_{\bold a}}$.
\end{enumerate}
\mn

Now
\mn
\begin{enumerate}
\item[$(*)_{5.9}$]   $(a) \quad$ let 
$\varepsilon \in \bbQ_M \subseteq \bbR_M$ be a true rational such that 

\hskip25pt $\xi(\varphi_2) > \varepsilon > \xi(\varphi_2)k_*/(k_*+1)+1/n(1)$
\sn
\item[${{}}$]  $(b) \quad$ let $d_* = \lfloor(a_*)^\varepsilon \rfloor
\in M_{\bold a}$ computed in $\bbR_{\bold a}$ and $c_* =
\lfloor(a_*)^{\varepsilon-1/n(1)}\rfloor$ 
\sn
\item[$(*)_{5.10}$]   in $M_{\bold a}$ we can define an
increasing sequence $\langle a_{1,i}:i < i(*)\rangle$, so $i(*)
\in M_{\bold a}$ such that
\begin{enumerate}
\item[${{}}$]   $\bullet \quad a_{1,0} = 0,a_{1,i(*)} = a_*$
\sn
\item[${{}}$]   $\bullet \quad a_{1,i+1} = \text{ min}
\{a:\varphi_1(a)$ and $a_{1,i} < a$ and $|\varphi_1(M_{\bold a})
\cap [a_{1,i},a)|$ is $d_*\}$
\sn
\item[${{}}$]   $\bullet \quad |\{a:\varphi_1(a)$ and $a_{1,i(*)-1}
\le a < a_*\}|$ is $\ge d_*$ but $\le 2d_*$
\end{enumerate}
\item[$(*)_{5.11}$]   $(a) \quad$ in $M_{\bold a}$ we can define

$\qquad u = \{i <i(*):M_{\bold a} \models \vartheta_3[a_{1,i},a_{1,i+1}]\}$

$\qquad v = \{i < i(*):i \notin u\}$
\sn
\item[${{}}$]  $(b) \quad$ let $\varphi_{1,i}(x) := \varphi_1(x) 
\wedge a_{1,i} \le x < a_{1,i+1}$.
\end{enumerate}
\mn

So (will be used in Case 1 below)
\mn
\begin{enumerate}
\item[$(*)_{5.12}$]   $(a) \quad \xi(\varphi_{1,i}(x)) = \varepsilon$
for $i<i(*)$
\sn
\item[${{}}$]  $(b) \quad$ if $i < i(*)$ and $i \in v$ then

\hskip25pt $M_{\bold a} \models ``|\varphi_2(M_{\bold a}) 
\cap (F(a_{1,i}),F(a_{1,i+1}))_{M_{\bold a}}| \ge
|\varphi_{1,i}(M_{\bold a})| \times a^{-1/n(1)}_*"$

\hskip25pt $= b_* \times a^{-1/n(1)}_*"$.
\end{enumerate}
\mn
[Why?  Clause (a) is obvious by the definition of $\xi(-)$ and
$a_{1,i+1}$.  For clause (b) note that by 
the definition of $\vartheta_3$ in $(*)_{5.5}(e)$ we have
$M_{\bold a} \models \neg \psi_3[a_{1,i},a_{2,i};F^{[\varphi_2]}(a_{1,i}),
F^{[\varphi_2]}(a_{2,i})]$, but by $(*)_{5.6}(a)$
we have $M_{\bold a} \models \psi_3[a_{1,i},a_{2,i};F(a_{1,i}),
F(a_{1,i})]$.  By the definition of $\psi_3$ in $(*)_{5.5}(d)$ we are done.]

Now towards Case 2 note
\mn
\begin{enumerate}
\item[$(*)_{5.13}$]  if $i_1 < i_2$ are from $u$ then $F(a_{1,i_1}) <
\sigma^M(a_{1,i_2},a_{1,i_2 +1})$.
\end{enumerate}
\mn
[Why?  Obvious by $(*)_{5.8}$.]
\mn
\begin{enumerate}
\item[$(*)_{5.14}$]   we define terms
$\sigma_1(x_1,x_2),\sigma_2(x_1,x_2) \in \Sigma^2_{\bold a}$ 
such that if $i<i(*)$ then:
\begin{enumerate}
\item[$(a)$]   $M_{\bold a} \models \sigma_1(a_{1,i},a_{1,i+1}) <
\sigma(a_{1,i},a_{1,i+1}) < \sigma_2(a_{1,i},a_{1,i+1})$
\sn
\item[$(b)$]   $\varphi_2(M_{\bold a}) \cap [\sigma^{M_{\bold a}}
(a_{1,i},a_{1,i+1}),\sigma^{M_{\bold
a}}_2(a_{1,i},a_{1,i+1}))_{M_{\bold a}}$ has $\le c_*$ elements, in
$M_{\bold a}$'s-sense
\sn
\item[$(c)$]   $\varphi_2(M_{\bold a}) \cap [\sigma^{M_{\bold a}}_1
(a_{1,i},a_{1,i+1}),\sigma^{M_{\bold a}}((a_{1,i},a_{1,i+1}))_{M_{\bold a}}$ has $\le c_*$ elements in
$M_{\bold a}$'s-sense
\sn
\item[$(d)$]   if $i < j$ are from $u$ then $M_{\bold a} \models 
``\sigma_2(a_{1,i},a_{1,i+1}) < \sigma(a_{1,j},a_{1,j+1})"$
\sn
\item[$(e)$]   $M_{\bold a} \models ``\sigma_1(a_{1,i},a_{1,i+1}) <
\sigma(a_{1,i},a_{1,i+1}) < \sigma_2(a_{1,i},a_{1,i+1})"$
\sn
\item[$(f)$]   if $i \in u$ then
$M_{\bold a} \models ``\sigma_1(a_{1,i},a_{1,i+1}) < F(a_{1,i}) <
\sigma(a_{1,i},a_{1,i+1}) < F(a_{1,i+1}) <
\sigma_2(a_{1,i},a_{1,i+1})"$.
\end{enumerate}
\end{enumerate}
\mn
[Why?  Let $\sigma_2(a_{1,i},a_{1,i+1})$ be maximal such that the
relevant part of (a) and (b),(d) holds and $\sigma_1(a_{1,i},a_{1,i
+1})$ be minimal such that the other part of (a) and $(c),(e)$ holds.
By $(*)_{5.8}$ we can finish.]

\noindent
Now the proof splits.
\bigskip

\noindent
\underline{Case 1}:  $M_{\bold a} \models ``|v|  \ge i(*)/3"$.  Here
we shall use clause $(\gamma)$ of $\boxplus_2(a)$.

Let $v_1 = \{i \in v:M_{\bold a} \models ``|\{j \in v:j<i\}|$ is
even$\}$, so $M_{\bold a} \models ``|v_1| \ge i(*)/6"$.  Let
$\varphi'_1(x) := \varphi_1(x) \wedge (\exists z)[z \in v_1 \wedge x
\in [a_{1,z},a_{1,z+1}) \wedge \neg(\exists y)(\varphi_2(y) \wedge y
\in [a_{1,z},a_{1,z+1}) \wedge y < x)]$.

Let $\varphi'_2(x) := \varphi_2(x) \wedge (\text{the number }
|\{y:\varphi_2(a) \wedge y < x\}|$ is divisible by $c_*$).

Now
\mn
\begin{enumerate}
\item[$(*)_{5.14}$]  $(a) \quad 
\bar\varphi' := (\varphi'_1(x),\varphi'_2(x)) \in \bbP$
\sn
\item[${{}}$]  $(b) \quad \xi(\varphi'_1(x)) = 
\xi(\varphi_1(x)) - \varepsilon$
\sn
\item[${{}}$]  $(c) \quad \xi(\varphi'_2(x)) = 
\xi(\varphi_2(x)) - \varepsilon +1/n(1)$.
\end{enumerate}
\mn
So 

\begin{equation*}
\begin{array}{clcr}
\xi(\varphi'_1(x))/\xi(\varphi'_2(x)) &= (\xi(\varphi_1(x))
-\varepsilon)/(\xi(\varphi_2(x) - \varepsilon +1/n(1)) \\
  &\ge (\xi(\varphi_1(x) - \xi(\varphi_2(x))/(\xi(\varphi_2) -
\xi(\varphi_2)k_*/(k_* +1)) \\
  &= (k_* +1)(\xi(\varphi_1(x))-\xi(\varphi_2(x)))/\xi(\varphi_2) > 2
\end{array}
\end{equation*}

\mn
and we fall under $(*)_{5.1}$ finishing the proof of $\boxplus_5$.
\bigskip

\noindent
\underline{Case 2}:  $M_{\bold a} \models ``|u| \ge i(*)/3$.

Define in $M_{\bold a}$

$u_2 = \{i \in u:``|\{j \in \varphi_1(M):j<i\}|$ is even"$\}$.

So $M_{\bold a} \models ``|u_1| \ge i(*)/6"$.  Now $M_{\bold a}$
satisfies 

\begin{equation*}
\begin{array}{clcr}
|\varphi_1(M)| \le (i(*)+1)d_* &\le 7
|\cup\{\varphi_1(M_{\bold a}) \cap [a_{1,i},a_{1,i+1}):i \in u_1\}| \\
  &= 7 \Sigma\{|\varphi_1(M_{\bold a}) \cap [a_{1,i},a_{1,i+1})|:
i \in u_1\}| \\
  &\le 7 \Sigma\{|\varphi_2(M_{\bold a}) \cap
[\sigma_1(a_{1,i},a_{1,i+1}),\sigma_2(a_{1,i},a_{1,i+1}))| \times
a^{1/n(1)}_*:i \in u_1\}| \\
  &7|\cup\{\varphi_2,M_{\bold a}) \cap
[\sigma_1(a_{1,i},a_{2,i+1}),\sigma_2(a_{1,i},a_{1,i+1})):
i \in u_2\}| \times a^{1/n(1)}_* \\
  &< 7 \times |\varphi_2(M_{\bold a})| \times a^{1/n(1)}_*.
\end{array}
\end{equation*}

\noindent
Together $|\varphi_2(M)| \le 7 \times |\varphi_2(M)| \times a^{1/n(1)}_{a_*}$.
But as $n(1)$ was chosen large enough, i.e. $1/n(1) <
\xi(\varphi_2(M_{\bold a}) - \xi(\varphi_2(M_{\bold a}))$,
contradiction.  So we are done proving $\boxplus_5$.
\bigskip

\noindent
\underline{Stage D}:

We define the sets $\bold T = \bold T_1 \cup \bold T_2 \cup \bold T_3$
of tasks where
\mn
\begin{enumerate}
\item[$\boxplus_6$]  $\bullet \quad \bold T_1 = \{(1,\varphi(x)):\varphi(x) \in
\Phi\}$, toward completeness
\sn
\item[${{}}$]  $\bullet \quad \bold T_2 
= \{(2,\sigma(x),\varphi(x)):\sigma(x) \in
\Sigma_{\bold a}$ and $p(x) \in \Gamma_{\bold a}\}$, toward 

\hskip25pt preserving $p(x)$ is omitted
\sn
\item[${{}}$]  $\bullet \quad \bold T_3 = \{(3,\sigma(x)):\sigma(x) \in
\Sigma_{\bold a}\}$, toward ``stopping $F$", ``omitting 

\hskip25pt the new type".
\end{enumerate}
\mn
Clearly $\bold T$ is countable, let $\langle \bold s_n:n <
\omega\rangle$ list it.

We now choose $\bar\varphi^n$ by induction on $n$ such that:
\mn
\begin{enumerate}
\item[$\boxplus_7$]  $(a) \quad \bar\varphi^n \in \bbP$
\sn
\item[${{}}$]  $(b) \quad \bar\varphi^m \le \bar\varphi^n$ for $m < n$
\sn
\item[${{}}$]  $(c) \quad$ if $n=m+1$ and $\bold s_m = (1,\varphi(x))$
\then \, $\varphi_n(x) \vdash \varphi(x)$ or $\varphi_n(x) \vdash 
\neg\varphi(x)$
\sn
\item[${{}}$]  $(d) \quad$ if $n=m+1$ and $\bold s_m =
(2,\sigma(x),p(x))$ \then \, for some $k,\varphi_n(x) \vdash ``\sigma(x)$

\hskip25pt $\notin (a_{p,k},b_{p,k})"$
\sn
\item[${{}}$]  $(e) \quad$ if $n=m+1$ and $\bold s_m = (3,\sigma(x))$
then for some $a_{1,m} < a_{2,m} \le a_*$ we 

\hskip25pt have $\varphi'_n(x)
\vdash a_{1,m} \le x < a_{2,m} \wedge \neg(F(a_{1,n}) \le \sigma(x) <
a_{2,m})$. 
\end{enumerate}
\mn
Why can we carry the induction?  $\bar\varphi^0$ trivial, for
$\bar\varphi^{m+1}$ if $\bold s_m \in \bold T_1$ by $\boxplus_3$, if
$\bold s_m \in \bold T_2$ by $\boxplus_4$ and if $\bold s_m \in \bold
T_3$ by $\boxplus_5$, more fully, let $\bold s_m = (3,\sigma(x))$, let
$\sigma'(x)$ be defined by $\sigma'(x) = \text{ min}\{y:y = b_*$ or
$\sigma(x) \le y \wedge \varphi(y)\}$, now apply $\boxplus_5$ to
$(\bar\varphi^m,\sigma'(-))$. 

Note
\mn
\begin{enumerate}
\item[$(*)_{7.1}$]  $\{\varphi_n(x):n < \omega\}$ is a complete type over
$M_{\bold a}$.
\end{enumerate}
\mn
[Why?  By clause (c) of $\boxplus_7$ and the choice of $\bold T_1$.]

By compactness there are $N$ and $c_*$ such that
\mn
\begin{enumerate}
\item[$(*)_{7.2}$]  $(a) \quad M_{\bold a} \prec N$
\sn
\item[${{}}$]  $(b) \quad c_*$ realizes $\{\varphi_n(x):n < \omega\}$.
\end{enumerate}
\mn
As $T$ being a completion of PA has definable Skolem functions, \wilog
\,
\mn
\begin{enumerate}
\item[$(*)_{7.3}$]  $N$ is the Skolem hull of $M_{\bold a} \cup \{c_*\}$.
\end{enumerate}
\mn
Now by $\boxplus_7(d)$
\mn
\begin{enumerate}
\item[$(*)_{7.4}$]  $N$ omit every $p \in \Gamma_{\bold a}$.
\end{enumerate}
\mn
Also by $\boxplus_7(c)$
\mn
\begin{enumerate}
\item[$(*)_{7.5}$]  $N$ omits $\{F(a) < x < F(a_1):a_0 < c_* < a_1
\le a_*$ and $\{a_0,a_1\} \subseteq M_{\bold a}\}$.
\end{enumerate}
\mn
By renaming, \wilog \, the set of elements of $N$ is a countable
ordinal, so we can finish the proof of the claim.
\end{PROOF}

\begin{discussion}
\label{d34}
1) Note that a natural approach is to approximate the type of $c_*$ by
formulas $\varphi(x)$ with parameters from $M$ such that
$\varphi(x) \vdash ``x < a_*"$ and $M_{\bold a}$ ``think"
$|\varphi(M_{\bold a})|$ is large enough than $b_*$.  Then for $\sigma(x) \in
\Sigma_{\bold a}$  (i.e. a term with parameters from $M$), which maps
$\{c:c <_M a_*\}$ into $\{d:d <_M b_*\}$ we have to ensure $\sigma(c_*)$
will not realize the undesirable type, so it is natural to ``shrink"
$\varphi(x)$ to $\varphi'(x)$ such that ``$|\varphi'(M)|$ is large
enough then $|\varphi(M)|/b_*"$ in 
the sense of $M$ and $\sigma(-)$ is constant on
$\varphi'(M)$.  This requires that also $|\varphi(M)|/b_*$ is large so
a natural choice for
``$\varphi(x)$ is large" means, e.g. $M \models ``|\varphi(M)| \ge
b^n_*"$ for every $n \in \bbN$.  This is fine if $\neg(a_* E^4_M
b_*)$, but not if we know $\neg(a_* E^3_M b_*)$ but possibly $a_*
E^4_M b_*$.  This motivates the main definition of $\bbP$ below.

\noindent
2) We shall say that $(\varphi_1,\varphi_2) \in \bbP$ is a 
``poor man's substitute" to the    original problem \when \,:
\mn
\begin{enumerate}
\item[$(a)$]  $[0,a_*)_{M_{\bold a}}$ is replaced by $\varphi_1(M)$ 
\sn
\item[$(b)$]  $[0,b_*)_{M_{\bold a}}$ is replaced by $\varphi_2(M)$ 
\sn
\item[$(c)$]  $F \rest [0,a_*)_{M_{\bold a}}$ is replaced by $F \rest 
\varphi_1(M_{\bold a})$, really rounded to $\varphi_2(M)$
\sn
\item[$(d)$] $a_* > b_* \wedge \neg(a_* E^3_{M_{\bold a}},b_*)$ is 
replaced by $\xi(\varphi_1) > \xi(\varphi_2)$, see $\boxplus_1(e)$ in
the proof.
\end{enumerate}
\mn
3) Why we demand 
condition $(\delta)$ of $\boxplus_2(a)$ in the proof of \ref{d31}?

Assume $M_{\bold a} \prec M$ and $a \in M \backslash M_{\bold a},a <_M
a_*$ and $F^+$ is an automorphism of $M \rest \{<\}$ extending $F$
\then \, $\{f(a_1):a_1 <_M a_*,a_1 <_M a\},\{F(a_2):a_2 \in 
M_{\bold a},a_* < a_2 < a_*\})$ is a cut of $M_{\bold a}$,
which $F^+(a)$ realizes in $M$.  If $M$ ``thinks" $|\varphi_2(M)|$ is $\ll
b_*,F$ may be one-to-one from $\varphi_2(M)$ onto some definable
subset $\varphi'_2(M) \subseteq [0,b_*)_M$.  A reasonable suggestion
is to demand $|\varphi_2(M)| \gg b_*$.  Consider for transparency the
case $M_{\bold a} \models ``a_* < b_* b_*"$.

But then let $E$ be the definable convex equivalent relation on
$[0,a_*)$ such that each equivalence class is of size $b_{**}$, 
then the cut the new element realizes is really a cut of $[0,a_*)/E$.
Now $F^+$ maps every
$E$-equivalence class to some $E'$-equivalence class, $E'$ a definable
convex equivalence relation on $[0,b_{**})$ and $F$ as a map from
$[0,a_*)/E$ into $[0,b_*)/E'$ is defined, possible if $|[0,a_*)/E| =
|[0,b_*)/E'|$. 

The solution is via clause $(\delta)$, which tells us that in part
(1),(2) of the discussion, $\xi(\varphi_1) > \xi(\varphi_2)$ is a
real substitute, see clause (d) in part (2).

\noindent
4) Why clause $(\gamma)$ in $\boxplus_2(a)$, defining $\bbP$?
Otherwise $\varphi_2(-)$ may be irrelevant to the type we like to 
omit, so impossible.

\noindent
5) By such approximations, i.e. member of $\bbP$,
\mn
\begin{enumerate}
\item[$(A)$]  why can we arrive to a complete type?
\end{enumerate}
\mn
\underline{Answer}:  
As if we divide $\varphi_1$ to two sets at least one has the same $\xi(-)$:
\mn
\begin{enumerate}
\item[$(B)$]   why can we continue to omit $p(x) \in \Gamma_{\bold a}$?
\end{enumerate}
\mn
\underline{Answer}: 
As if $\sigma(-)$ is a definable (in $M_{\bold a}$) function 
with domain $\varphi_1$ let
$d_*$ be maximal such that $|\{a \in \varphi_1(M):\sigma(a) < d_*\}|
\le \frac 12|\varphi_2(M)|$, i.e. is in the middle in the right sense.

If $\sigma^{-1}\{d_*\}$ is large enough we easily finish; otherwise for some
$n$ we have $d_* \notin (a_{p,n},b_{p,n})$, so $\varphi_1(M) \wedge \sigma(x)
\notin (a_{p,n},b_{p,n})$ has $\ge \frac 12(\varphi_2(M))$ elements
\mn
\begin{enumerate}
\item[$(C)$]   why can guarantee that such $\sigma(x)$ does not realize
the forbidden new type?
\end{enumerate}
\mn
\underline{Answer}:
This is a major point.  
If $\xi(\varphi_1) > 2 \xi(\varphi_2)$ this is easy (as in the case we
use $\neg a_* E^4_M b_*$) and if for some $a_1 < a_2$ we have
$\xi(\varphi'_2) > \xi(\varphi'_2)$ we let

\[
\varphi'_2(x) = (\varphi_2(x) \wedge a_1 \le x < a_2 \wedge \sigma(x) \notin
(F(a_1),F(a_2))
\]

\mn
and we let

\[
\varphi''_2(x) := (\varphi_2 \wedge F(a_1) \le x < F(a_2))
\]

\mn
we are done, so assume there are no such $a_1,a_2$.

We consider two possible reasons for the ``failure" of a suggested
pair $(a_1,a_2)$.  One reason is that maybe the length of the interval
$[F(a_1),F(a_2))$ of $\varphi_2(M_1)$ is too large.  The second is
that it is small enough but $\sigma(-)$ maps the large majority of
$\varphi_1(M) \cap [a_1,a_2)$ into $[F(a_1),F(a_2))$.  In the second
version we can define a version of it's property satisfied by
$(a_1,a_2,F(a_1),F(a_2))$.  So we have enough intervals of pseudo
second kind (pseudo means using the definable version of the property).
 So dividing $\varphi_1(M)$ to convex
subsets of equal (suitable) size (essentially
$a^{\xi(\varphi_2)}_*,\zeta \in \bbR_{>0}$ small enough) by $\langle
a_i:i < i(*)\rangle$ we have: for some such interval $[a_i,a_{i+1})$ there are
$b_1,b_2$ as above.  For those for which we cannot define
$(F(a_1),F(a_2))$ we can define it up to a good approximation.  If
there are enough, (this may include ``pseudo cases" in respect to $F$)
we can replace $\varphi_1(M)$ by $\varphi'_1(M) = \{a_i:i < i(*)\}$
and $\varphi'_2(-)$ defined by the function above.

So $|\varphi'(M)|$ is significantly smaller than $|\varphi_1(M)|$,
essentially $\xi(\varphi'_1) = \xi(\varphi_1) - \xi(a_{i+1}-a_i) \sim 
\xi(\varphi_1) - \xi(\varphi_2) + \zeta,\zeta$ quite small.  But we are
over-compensating so we decrease $\varphi_2(x)$ to
$\varphi'_2(x)$ which is quite closed to $\{F(a_i):[a_i,a_{i+1})$ is
of the pseudo second kind$\}$ and $\xi(\varphi'_2)$ is essentially
$\xi(\varphi_2) - \xi(\varphi_2) + \zeta \sim \zeta$.  So both lose
similarly in the $\xi(-)$ measure but now, if we have arranged the
numbers correctly $\xi(\varphi'_2) > 2 \xi(\varphi'_2)$, a case we know
to solve.

If there are not enough $i$'s of the pseudo second kind, the function
essentially inflates the image getting a finite cardinality arithmetic
contradiction. 
\end{discussion}

\begin{theorem}
\label{d35}
Assume $\diamondsuit_{\aleph_1}$.  If $M$ is a countable model of $\PA$, 
\then \, $M$ has an elementary extension $N$ of cardinality $\aleph_1$ such
that $E^5_N = E^3_N$, i.e. is 3-{\rm o.r.}
\end{theorem}

\begin{PROOF}{\ref{d35}}
\Wilog \, $M$ has universe a countable ordinal.
As we are assuming $\diamondsuit_{\aleph_1}$, we choose $F_\alpha$ a
partial function from $\alpha$ to $\alpha$ for $\alpha < \aleph_1$,
i.e. $\bar F = \langle F_\alpha,\alpha < \aleph_1\rangle$ such that
for  every partial function $F:\aleph_1 \rightarrow \aleph_1$, for 
stationarily many countable limit ordinals $\delta$ 
we have $F_\delta = F \rest \delta$.

We now choose $\bold a_\alpha \in \text{\rm AP}_{\aleph_0}$ by
induction on $\alpha < \aleph_1$ such that
\mn
\begin{enumerate}
\item[$(a)$]  $\bullet \quad M_{\bold a_0} = M$
\sn
\item[${{}}$]  $\bullet \quad \Gamma_{\bold a_0} = \emptyset$
\sn
\item[$(b)$]  $\langle \bold a_\beta:\beta \le \alpha\rangle$ is
$\le_{\text{\rm AP}}$-increasing continuous
\sn
\item[$(c)$]  if $\alpha = \delta +1,\delta$ is a countable limit ordinal,
$M_{\bold a_\delta}$ has universe $\delta$ and for some 
$a_\delta,b_\delta$ the tuple $(\bold
a_\delta,a_\delta,b_\delta,F_\delta)$ satisfies the assumptions of
\ref{d31} on $(\bold a,a_*,b_*,F)$, they are necessarily unique (see
\ref{d31}(A)(c)), \then \, $\bold a_{\delta +1}$ 
satisfies its conclusion (for some $c_\delta$).
\end{enumerate} 
\mn
Why can we carry the induction?

For $\alpha = 0$ recall clause (a).

For $\alpha =1$, as $\Gamma_{\bold a_0} = \emptyset$ let $M_{\bold
  a_1}$ be a countable model such that $M = M_{\bold a_0} \prec
  M_{\bold a_1},M \ne M_{\bold a_1}$ and \wilog \, the universe of
  $M_{\bold a_1}$ is a countable ordinal.

Lastly, let $\Gamma_{\bold a_1} = \emptyset$.

For $\alpha$ a limit ordinal use \ref{d5}(2), i.e. choose the union,
this is obvious.

For $\alpha = \beta +1$, if clause (c) apply use Claim \ref{d31}.

For $\alpha = \beta +1 >1$ when clause (c) does not apply, this is easier
than \ref{d31} (or choose $(a_*,h_*,F)$ such that $(\bold
a_\beta,a_*,b_*,F)$ are as in the assumption \ref{d31}, this is
possible because $M_{\bold a_\beta}$ is non-standard, see the case
$\alpha =1$,  and note that $a,b \in
M_{\bold a_\beta} \backslash \bbN \Rightarrow a E^5_{M_{\bold a}} b$
because $M_{\bold a}$ is countable; so we can use \ref{d31}).

Having carried the induction let $N = \cup\{M_{\bold a_\alpha}:\alpha <
\aleph_1\}$.

Clearly $N$ is a model of $T$ of cardinality $\aleph_1$.  We know that
$E^3_N \subseteq E^5_N$ by \ref{b9}.  Toward contradiction assume
$a_* E^5_N b_*$ but $\neg(a_* E^3_N b_*)$ where $a_*,b_* \in N
\backslash M$.  \Wilog \, $b_* < a_*$ and let $F$ be an
order-isomorphism from $N_{< a_*}$ onto $N_{< b_*}$.  So $S =
\{\delta:F \rest \delta = F_\delta\}$ is stationary and $E =
\{\delta:a_*,b_* \in M_{\bold a_\delta},M_{\bold a_\delta}$ has
universe $\delta$ and $F$ maps $M^{\bold a_\delta}_{< a_*}$ onto
$M^{\bold a_\delta}_{< b_*}\}$ is a club of $\aleph_1$. 

Choose $\delta \in S \cap E$ and use the choice of $\bold a_{\delta
  +1}$, i.e. clause (c) to get a contradiction.
\end{PROOF}

\begin{theorem}
\label{d39}
Assume $\lambda = \lambda^{< \lambda}$ and
$\diamondsuit_S$ where $S = S^{\lambda^+}_\lambda = \{\delta <
\lambda^+:\cf(\delta) = \lambda\}$.

For any model $M$ of $\PA$ there is a
$\lambda$-saturated model $N$ of $\Th(M)$ of cardinality $\lambda^+$ such
that $E^5_N \subseteq E^4_N$.
\end{theorem}

\begin{PROOF}{\ref{d39}}
Similar to \ref{d35} only the parallel of \ref{d31} is much easier.
\end{PROOF}

\begin{conjecture}
\label{d43}
1) Assume $\lambda$ is strong limit singular of cofinality $\aleph_0$ and
$\diamondsuit_S$ where $S = S^{\lambda^+}_{\aleph_0} = \{\delta
< \lambda^+:\cf(\delta) = \aleph_0\}$ and $\square_\lambda$.  
If $M$ is a model of $\PA$, then $\Th(M)$ has a $\lambda$-universal 
model $N$ of cardinality $\lambda^+$ which is 3-{\rm o.r.} 

\noindent
2) Any model $M$ of PA has a 3-{\rm o.r.} elementary extension.
\end{conjecture}
\newpage

\section {Weaker version of PA}

We may wonder what is the weakest version of PA needed in the various
results so below we define some variants and then remark when they
suffice.  But say when we add the function $2^x$, we prefer to add to
the vocabulary a new function symbol and the relevant axioms (rather
than an axiom stating that some definition of it has those
properties).  So we shall comment what version of PA is needed in the
results of \S1,\S2.

\begin{convention}
\label{k0}
A model is a model of PA$_{-4}$ (see below) of vocabulary
$\tau_{\PA}$ if not said otherwise.
\end{convention}

\begin{definition}
\label{k3}
We define the first order theories PA$_\ell$ for 
$\ell \in\{-1,\dotsc,-4\}$ and let PA$^{\com}_\ell$ 
be the set of completions of PA$_\ell$.

Let PA$_\ell$ consist of the following first order sentences 
in the vocabulary $\{0,1,<,+,\times\}$ of $\bbN$:
\mn
\begin{enumerate}
\item[$(a)$]   for $\ell \le 4$, the obvious axioms of addition and
product and order, that is axioms describing the non-negative parts of
discrete ordered rings,
\sn
\item[$(b)$]   if $\ell \le 3$ we also add division with remainder
by any $n \in \bbN$,
\sn
\item[$(c)$]   if $\ell \le 2$ also add division with remainder,
\sn
\item[$(d)$]   if $\ell \le 1$, we add a unary function $F_2$ written
$2^x$ with the obvious axioms for $x \mapsto 2^x$, including $(\forall
x)(\exists y)(2^y \le x < 2^{y+1})$.
\end{enumerate}
\end{definition}

\begin{claim}
\label{k5}
For $M,N$ are models of {\rm PA}$_{-4}$; we still have 
\ref{a7}(1),(1A),(1B),(3),(4) and \ref{a9} and
\ref{b3}(1),(2),(3),(5) and \ref{a19}(1),(2),(3). 
\end{claim}

\begin{claim}
\label{k7}
Claim \ref{a7}(2) holds \when \, $M$ is a model of {\rm PA}$_{-3}$.
\end{claim}

\begin{PROOF}{\ref{k7}}  
The only difference is why can we choose $c_1,c_2$ there?

Now if we assume
$M \models$ PA this is obvious, but we are assuming $M \models$ PA$_{-3}$,
still we can divide $b-a$ by $n-1$ and then get $c_1$ and $m < n-1$ such
that $b-a = (n-1) \times c_2 + m$.  Let $c_2 = a-c_1$ so $b=a + (n-1)
\times c_2 + m = c_1 + n \times c_2 + m$.  We still have to justify
using $a-c_2$, i.e. showing $c_2 \le a$, but otherwise $b-a = (n-1)
\times c_2 + m \ge (n-1) \times a +m$, i.e. $b \ge n \times a +m$,
contradiction. 
\end{PROOF}

\begin{theorem}
\label{k13}
If $M_1,M_2$ are models of {\rm PA}$_{-1}$ then \ref{a13} holds,
i.e. if $M_2$ is 2-order-rigid and $M_1,M_2$ are order-isomorphic
\then \, $M_1,M_2$ are almost $\{<,+\}$-isomorphic.
\end{theorem}

\begin{PROOF}{\ref{k13}}
As in \ref{a13} with the following minor additions:
\mn
\begin{enumerate}
\item[$\bullet$]  in the proof of $(*)_3$ we use $M_2 \models \PA_{-4}$,
\sn
\item[$\bullet$]  in the proof of $(*)_5$ we use $M_\ell \models 
\text{ PA}_{-2}$,
\sn
\item[$\bullet$]  in the proof of $(*)_0$ we use $M_2 \models \PA_{-2}$,
\sn
\item[$\bullet$]  in the proof of $(*)_7$ we use $M_1 \models \PA_{-4}$,
\sn
\item[$\bullet$]  in the proof of $(*)_8,\circledast_2$ we
use $M_\ell \models \PA_{-1}$.
\end{enumerate}
\end{PROOF}

\begin{claim}
\label{k14}
1) If $M \models \PA_{-3}$ then \ref{a19}(4) holds.

\noindent
2) If $M \models \PA_{-1}$ \then \, \ref{a19}(5) holds.
\end{claim}

\begin{PROOF}{\ref{k14}}
Straightforward.
\end{PROOF}

\noindent
We may wonder (see Definition \ref{a6}(2)).
\begin{question}
\label{a11}  Is there a 2-o.r. model of PA$_{-1}$?
\end{question}
\bigskip\bigskip

%\bibliographystyle{alphacolon}
%\bibliography{lista,listb,listx,listf,liste,listz}

\begin{thebibliography}{}

\bibitem[GT06]{GbTl05}
R\"udiger G\"obel and Jan Trlifaj, \emph{{Approximations and endomorphism
  algebras of modules}}, de Gruyter Expositions in Mathematics, vol.~41, Walter
  de Gruyter, Berlin, 2006.

\bibitem[KS06]{KoSc06}
R.~Kossak and J.~Schmerl, \emph{{The structure of models of Peano arithmetic}},
  Oxford University Press, 2006.

\bibitem[Sh:384]{Sh:384}
Saharon Shelah, \emph{{Compact logics in ZFC : Complete embeddings of atomless
  Boolean rings}}, {Non structure theory, Ch X}.

\bibitem[Sh:757]{Sh:757}
\bysame, \emph{{Quite Complete Real Closed Fields}}, Israel Journal of
  Mathematics \textbf{142} (2004), 261--272, math.LO/0112212.

\end{thebibliography}

\end{document}